\documentclass[11pt,a4paper]{article}
\usepackage{}
\usepackage{amsfonts}
\usepackage{mathrsfs}
\usepackage{multirow}
 \usepackage{epsfig}
 \usepackage{epstopdf}
\usepackage[T1]{fontenc}
\usepackage{geometry}
\usepackage{amsbsy,amsmath,latexsym,amsfonts, epsfig, color, authblk, amssymb, graphics, bm}
\usepackage{epsf,slidesec,epic,eepic}
\usepackage{fancybox}
\usepackage{fancyhdr}
\usepackage{setspace}
\usepackage{nccmath}
\usepackage{cases}
\usepackage[colorlinks, citecolor=blue]{hyperref}

\newtheorem{theorem}{Theorem}[section]
\newtheorem{lemma}{Lemma}[section]

\newtheorem{proposition}{Proposition}[section]

\newtheorem{remark}{Remark}[section]

\newenvironment{proof}{{\noindent \bf Proof:}}{\hfill$\Box$\medskip}

\definecolor{lred}{rgb}{1,0.8,0.8}
\definecolor{lblue}{rgb}{0.8,0.8,1}
\definecolor{dred}{rgb}{0.6,0,0}
\definecolor{dblue}{rgb}{0,0,0.5}
\definecolor{dgreen}{rgb}{0,0.5,0.5}

 \title{Regular and limiting normal cones to the graph of the subdifferential mapping of the nuclear norm\!
 \footnote{Supported by the National Natural Science Foundation of China under project No.11571120
 and the Natural Science Foundation of Guangdong Province under project No. 2015A030313214.}}

\author{Yulan Liu\footnote{Ylliu@gdut.edu.cn. School of Mathematics, South China University of Technology//
 School of Applied Mathematics, Guangdong University of Technology, Guangzhou (510641), China.}\ \ {\rm and}\ \
 Shaohua Pan\footnote{Corresponding author (shhpan@scut.edu.cn). School of Mathematics, South China University of Technology,
 Guangzhou (510641), China.}}

 \date{January 20, 2017 (revision)}

\begin{document}

 \maketitle

 \begin{abstract}
  This paper focuses on the characterization for the regular and limiting normal cones
  to the graph of the subdifferential mapping of the nuclear norm, which is essential to
  derive optimality conditions for the equivalent MPEC (mathematical program with equilibrium constraints)
  reformulation of rank minimization problems.
 \end{abstract}



 \medskip
%----------------------------------------------------------------------------------Introduction
 \section{Introduction}\label{sec1}

  Let $\mathbb{Z}$ be a finite dimensional vector space, and $\mathbb{R}^{m\times n}$ be
  the vector space of all $m\!\times\!n$ real matrices equipped with the trace inner product
  $\langle \cdot,\cdot\rangle$ and its induced norm $\|\cdot\|_F$. Denote by $\|X\|_*$
  the nuclear norm of a matrix $X\in\mathbb{R}^{m\times n}$, i.e., the sum of all singular
  values of $X$, and by $\|X\|$ the spectral norm of $X$. Consider the following optimization problem
  \begin{equation}\label{prob}
   \min_{z\in\mathbb{Z}}\Big\{f(z)\!:\ g(z)\in K,\ (G(z),H(z))\in{\rm gph}\,\partial\|\cdot\|_*\Big\},
  \end{equation}
  where $f\!:\mathbb{Z}\to\mathbb{R}$ is a locally Lipschitz function,
  $g\!:\mathbb{Z}\to\mathbb{R}^l\times\mathbb{R}^{m\times n}$ and
  $G,H\!:\mathbb{Z}\to\mathbb{R}^{m\times n}$ are continuously differentiable mappings,
  $K$ is a simple closed convex set of $\mathbb{R}^l\times\mathbb{R}^{m\times n}$,
  and ${\rm gph}\,\partial\|\cdot\|_*$ denotes the graph of the subdifferential mapping
  of the nuclear norm.

  \medskip

  By the proof of Lemma \ref{graph-subdiff} below, we know that $(G(z),H(z))\in{\rm gph}\,\partial\|\cdot\|_*$
  if and only if $H(z)\!\in\mathop{\arg\max}_{Y\in\mathbb{B}}\langle G(z),Y\rangle$,
  where $\mathbb{B}\!:=\!\big\{Z\in\mathbb{R}^{m\times n}\ |\ \|Z\|\le 1\big\}$ in this paper.
  This shows that the constraint $(G(z),H(z))\!\in{\rm gph}\,\partial\|\cdot\|_*$
  represents a kind of optimality conditions. Therefore, problem \eqref{prob}
  is a mathematical program with a matrix equilibrium constraint
  $(G(z),H(z))\!\in{\rm gph}\,\partial\|\cdot\|_*$, which extends the optimization
  problems with polyhedral variational inequality constraints \cite{Ye99,Ye00},
  second-order cone complementarity constraints \cite{Ye15,ZZW11},
  or positive semidefinite (PSD) complementarity constraints \cite{DingYS14,WZZhang14}
  to those with general matrix equilibrium constraints.

  \medskip

  Our interest in \eqref{prob} comes from the fact that it covers an equivalent
  reformulation of low-rank optimization problems. Indeed, for the following rank minimization problem
  \begin{equation}\label{rank-min}
    \min_{X\in\mathbb{R}^{m\times n}}\!\Big\{{\rm rank}(X)\!:\ \|\mathcal{A}(X)-b\|\le\delta,\,X\in\Omega\Big\},
  \end{equation}
  from \cite[Section 3.1]{Bi14} and Lemma \ref{graph-subdiff} in Section \ref{sec2}
  we know that it can be reformulated as
  \begin{equation}\label{equiv-rank}
   \min_{X,Y\in\mathbb{R}^{m\times n}}\!\Big\{\|Y\|_*\!:\ \|\mathcal{A}(X)-b\|\le\delta,\,
   X\in\Omega,\,(X,Y)\in{\rm gph}\,\partial\|\cdot\|_*\Big\},
  \end{equation}
  where $\mathcal{A}\!: \mathbb{R}^{m\times n}\to\mathbb{R}^{N}$ is a sampling operator,
  $b\in\mathbb{R}^{N}$ is a noisy observation vector, $\delta>0$ is a constant related to
  the noise level, and $\Omega\subseteq\mathbb{R}^{m\times n}$ is a closed convex set.
  Clearly, problem \eqref{equiv-rank} is a special case of \eqref{prob}
  with $K=\{x\in\mathbb{R}^N\!:\|x\|\le\delta\}\times\Omega$, and
  $g(z)=(\mathcal{A}(X)-b; X)$ and $(G(z),H(z))=(X,Y)$ for
  $z=(X,Y)\in\mathbb{R}^{m\times n}\times\mathbb{R}^{m\times n}$.

  \medskip

  As it is well known, low-rank optimization problems have wide applications in
  many fields such as statistics \cite{Negahban11}, system identification
  and control \cite{Fazel02,FPST13}, signal and image processing \cite{ChenChi14},
  machine learning \cite{Srebro04}, finance \cite{Pietersz04}, quantum tomography \cite{Gross11},
  and so on. This motivates us to develop the optimality conditions and stability
  results for problem \eqref{prob}, especially the equivalent reformulation
  \eqref{equiv-rank} of the rank minimization problem \eqref{rank-min}. To achieve this goal,
  an essential step is to provide the characterization for the regular and limiting normal cones
  to the graph of the subdifferential mapping of the nuclear norm. In this work,
  we shall resolve this critical problem and as a byproduct establish the (regular) coderivative of the projection operator onto
  the unit spectral norm ball.

  \medskip

  Throughout this paper, we stipulate $m\le n$. Let $\mathbb{S}^m$ be the space of all $m\times m$
  real symmetric matrices and $\mathbb{S}_{+}^m$ be the cone of all PSD matrices from $\mathbb{S}^m$.
  Let $\mathbb{O}^{n_1\times n_2}$ be the set of all $n_1\times n_2$ real matrices with orthonormal columns,
  and $\mathbb{O}^{n_1}$ be the set of all $n_1\times n_1$ real orthogonal matrices.
  For $Z\in\!\mathbb{R}^{m\times n}$, $Z_{\alpha\beta}$ denotes the submatrix consisting of
  those $Z_{ij}$ with $(i,j)\in\alpha\times\beta$. Let $e$ and $E$ be the vector and
  the matrix of all ones whose dimensions are known from the context,
  and for a vector $z$, ${\rm Diag}(z)$ denotes a diagonal matrix
  which may be square or rectangular. For a given set $S$, $\delta_S(\cdot)$ means
  the indicator function over $S$; $\mathcal{T}_S(x)$ and $\mathcal{T}_S^{i}(x)$ denote
  the tangent cone and the inner tangent cone to $S$ at $x$, respectively; and $\mathcal{N}^{\pi}_S(x)$,
  $\widehat{\mathcal{N}}_S(x)$ and $\mathcal{N}_S(x)$ denote the proximal normal
  cone, the regular normal cone and the limiting normal cone to $S$ at $x$, respectively
  (see \cite{RW98,Mordu06}).
%------------------------------------------------------------------------------------------Section 2
  \section{Preliminaries}\label{sec2}

  This section includes three technical lemmas used for the subsequent analysis.
  The first one gives some characterizations for the graph of
  the subdifferential mapping $\partial\|\cdot\|_*$.
%--------------------------------------------------------------------------------------------
 \begin{lemma}\label{graph-subdiff}
  The graph of the subdifferential mapping $\partial\|\cdot\|_*$ has the following forms
  \begin{align*}
   {\rm gph}\,\partial\|\cdot\|_*\!={\rm gph}\,\mathcal{N}_{\mathbb{B}}^{-1}
    &\!=\!\big\{(X,Y)\in\mathbb{R}^{m\times n}\times \mathbb{R}^{m\times n}\ |\
    \|X\|_*-\langle X,Y\rangle=0,\,\|Y\|\le 1\big\}\\
   &\!=\big\{(X,Y)\in\mathbb{R}^{m\times n}\times \mathbb{R}^{m\times n}\ |\
    \Pi_{\mathbb{B}}(X+Y)=Y\big\}.
  \end{align*}
 \end{lemma}
 \begin{proof}
  Notice that $(X,Y)\in{\rm gph}\,\partial\|\cdot\|_*$ if and only if
  $Y\in\partial\|X\|_*=\partial\delta_{\mathbb{B}}^*(X)$, where $\delta_{\mathbb{B}}^*$
  is the conjugate function of $\delta_{\mathbb{B}}$. By \cite[Theorem 23.5]{Roc70},
  $\partial\delta_{\mathbb{B}}^*(X)=(\partial\delta_{\mathbb{B}})^{-1}(X)=\mathcal{N}_{\mathbb{B}}^{-1}(X)$.
  This shows that $(X,Y)\in{\rm gph}\,\partial\|\cdot\|_*$ iff
  $(X,Y)\in{\rm gph}\,\mathcal{N}_{\mathbb{B}}^{-1}$, and the first equality follows.
  Since $(X,Y)\in{\rm gph}\,\mathcal{N}_{\mathbb{B}}^{-1}$ if and only if
  $X\in\mathcal{N}_{\mathbb{B}}(Y)$, which is equivalent to saying that
  $Y\in\mathop{\arg\max}_{Z\in\mathbb{B}}\langle X,Z\rangle$ or equivalently
  $\langle X,Y\rangle=\|X\|_*$ and $\|Y\|\le 1$. The second equality follows.
  For the third equality, by \cite[Proposition 6.17]{RW98} $\Pi_{\mathbb{B}}(X+Y)=Y$
  iff $X\in\mathcal{N}_{\mathbb{B}}(Y)$, which means that
  ${\rm gph}\,\partial\|\cdot\|_*=\big\{(X,Y)\in\mathbb{R}^{m\times n}\times \mathbb{R}^{m\times n}
  \ |\ \Pi_{\mathbb{B}}(X+Y)=Y\big\}$.
 \end{proof}

  Lemma \ref{graph-subdiff} shows that the graph of $\partial\|\cdot\|_*$ has an intimate
  link with $\Pi_{\mathbb{B}}(\cdot)$. Motivated by this, we here establish an important
  property for the projection operator $\Pi_{\mathbb{B}}(\cdot)$.
 %--------------------------------------------------------------------------------------------Lemma 3.1
  \begin{lemma}\label{Calmly-Bdiff}
   The projection map $\Pi_{\mathbb{B}}$ is calmly B-differentiable
   at any given $X\in\mathbb{R}^{m\times n}$, i.e.,
   \(
     \Pi_{\mathbb{B}}(X+H)-\Pi_{\mathbb{B}}(X)-\Pi_{\mathbb{B}}'(X;H)=O(\|H\|^2)
   \)
   for any $\mathbb{R}^{m\times n}\ni H \to 0$.
  \end{lemma}
  \begin{proof}
   Define $\psi(x)\!:={\rm mid}(-e,x,e)$ for $x\in\mathbb{R}^m$. The mapping
   $\psi$ is Lipschitz continuous everywhere. For any given $x\in\mathbb{R}^m$
   and any $h\in\mathbb{R}^m$, a simple calculation yields that
  \begin{equation}\label{gdir-derivative}
    \big[\psi'(x;h)\big]_i
    =\left\{\!\begin{array}{cl}
              0 & {\rm if}\ |x_i|>1,\\
             {\rm sign}(x_i)\min(0,{\rm sign}(x_i)h_i) & {\rm if}\ |x_i|=1,\\
               h_i & {\rm if}\ |x_i|<1
              \end{array}\right.\ \ {\rm for}\ i=1,2,\ldots,m.
  \end{equation}
  It is easy to check that $\psi(x+h)-\psi(x)-\psi'(x;h)=0$. Hence,
  $\psi$ is calmly B-differentiable in $\mathbb{R}^m$. Since $\psi$ is symmetric,
  i.e., $\psi(x)=Q^{\mathbb{T}}\psi(Qx)$ for any signed permutation matrix
  $Q\in\mathbb{R}^{m\times m}$ and $x\in\mathbb{R}^m$, the desired result
  follows by invoking \cite[Theorem 5.5]{DingSST14}.
 \end{proof}

  Next we give the expression of the directional derivative of $\Pi_{\mathbb{B}}$
  by \cite[Theorem 3.4]{DingSST14}.
%-----------------------------------------------------------------------------------------------------
  \begin{lemma}\label{dir-derivative}
  Let $\overline{Z}\in\mathbb{R}^{m\times n}$ have the SVD of the form
  $\overline{Z}=\overline{U}\big[{\rm Diag}(\sigma(\overline{Z}))\ \ 0\big]\overline{V}^{\mathbb{T}}$, where
   $\overline{U}\in\mathbb{O}^m$ and $\overline{V}=[\overline{V}_1\ \ \overline{V}_2]\in\mathbb{O}^{n}$
  with $\overline{V}_1\in\mathbb{R}^{n\times m}$. Define the index sets
  \begin{subnumcases}{}
  \label{ab-index}
  \alpha:=\big\{i\in\{1,\ldots,m\}\ |\ \sigma_i(\overline{Z})>1\big\},\ \beta:=\big\{i\in\{1,\ldots,m\}\ |\ \sigma_i(\overline{Z})=1\big\},\\
  \label{gc-index}
  \gamma:=\big\{i\in\{1,\ldots,m\}\ |\ \sigma_i(\overline{Z})<1\big\},\ c:=\big\{m\!+\!1,m\!+\!2,\ldots,n\big\}.
  \end{subnumcases}
  Let $\Omega_1,\Omega_2\in\mathbb{S}^{m}$ and $\Omega_3\in\mathbb{R}^{m\times (n-m)}$ be
  the matrices associated with $\sigma(\overline{Z})$, defined by
  \begin{align}\label{Omega1}
   \big(\Omega_1\big)_{ij}
    &:=\left\{\begin{array}{cl}
          \frac{\min(1,\sigma_i(\overline{Z}))-\min(1,\sigma_j(\overline{Z}))}
          {\sigma_i(\overline{Z})-\sigma_j(\overline{Z})}& {\rm if}\ \sigma_i(\overline{Z})\ne\sigma_j(\overline{Z}),\\
           0 & {\rm otherwise}\\
   \end{array}\right.\ \ i,j\in\{1,2,\ldots,m\},\\
   \label{Omega2}
   \big(\Omega_2\big)_{ij}
   &:=\left\{\begin{array}{cl}
          \frac{\min(1,\sigma_i(\overline{Z}))+\min(1,\sigma_j(\overline{Z}))}
          {\sigma_i(\overline{Z})+\sigma_j(\overline{Z})}& {\rm if}\ \sigma_i(\overline{Z})+\!\sigma_j(\overline{Z})\ne 0,\\
           0 & {\rm otherwise}\\
   \end{array}\right.\ \ i,j\in\{1,2,\ldots,m\},\\
   \label{Omega3}
   \big(\Omega_3\big)_{ij}
   &:=\left\{\begin{array}{cl}
          \frac{\min(1,\sigma_i(\overline{Z}))}{\sigma_i(\overline{Z})}
          & {\rm if}\ \sigma_i(\overline{Z})\ne 0,\\
           0 & {\rm otherwise}\\
   \end{array}\right.\ \ i\in\{1,\ldots,m\},j\in\{1,\ldots,n\!-\!m\}.
  \end{align}
  Then, for any $H\in\mathbb{R}^{m\times n}$, with $\widetilde{H}_1=\overline{U}^{\mathbb{T}}H\overline{V}_1$
  and $\widetilde{H}=[\overline{U}^{\mathbb{T}}H\overline{V}_1\ \ \overline{U}^{\mathbb{T}}H\overline{V}_2]$
  it holds that
  \begin{align*}
  \Pi_{\mathbb{B}}'(\overline{Z};H)
    =\!\overline{U}\left[\begin{matrix}
      (\Omega_2)_{\alpha\alpha}\circ \big(\mathcal{X}(\widetilde{H}_1)\big)_{\alpha\alpha}
      &(\Omega_2)_{\alpha\beta}\!\circ\big(\mathcal{X}(\widetilde{H}_1)\big)_{\alpha\beta}& \widehat{H}_{\alpha\gamma}
      &(\Omega_3)_{\alpha c}\!\circ\!\widetilde{H}_{\alpha c}\\
      (\Omega_2)_{\beta\alpha }\!\circ\! \big(\mathcal{X}(\widetilde{H}_1)\big)_{\beta\alpha }&\widetilde{H}_{\beta\beta} \!-\!\Pi_{\mathbb{S}_{+}^{|\beta|}}(\mathcal{S}(\widetilde{H}_{\beta\beta}))&
      \widetilde{H}_{\beta\gamma}&\widetilde{H}_{\beta c}\\
      \widehat{H}_{\gamma\alpha}&\widetilde{H}_{\gamma\beta}&\widetilde{H}_{\gamma\gamma}
      &\widetilde{H}_{\gamma c}\\
      \end{matrix}\right]\overline{V}^{\mathbb{T}},
  \end{align*}
   where $\widehat{H}_{ij}=(\Omega_1)_{ij}\big(\mathcal{S}(\widetilde{H}_1)\big)_{ij}+(\Omega_2)_{ij}\big(\mathcal{X}(\widetilde{H}_1)\big)_{ij}$
   for $(i,j)\in\alpha\times\gamma$ or $(i,j)\in\gamma\times\alpha$,
   and $\mathcal{S}\!:\mathbb{R}^{m\times m}\to\mathbb{S}^m$ and
   $\mathcal{X}\!:\mathbb{R}^{m\times m}\to\mathbb{R}^{m\times m}$ are two linear mappings
   defined by
  \begin{equation}\label{ST-oper}
   \mathcal{S}(Z):=(Z\!+\!Z^{\mathbb{T}})/2\ \ {\rm and}\ \ \mathcal{X}(Z):=(Z\!-\!Z^{\mathbb{T}})/2
   \quad\ \forall Z\in\mathbb{R}^{m\times m}.
  \end{equation}
  \end{lemma}

%------------------------------------------------------------------------------------------Section
 \section{Regular and limiting normal cones to ${\rm gph}\,\partial\|\cdot\|_*$}

 In this section we shall derive the exact formula for the regular and limiting
 normal cones to ${\rm gph}\,\partial\|\cdot\|_*$. First, we focus on the formula
 of the regular normal cone to ${\rm gph}\,\partial\|\cdot\|_*$.

%-------------------------------------------------------------------------------------------Subsection 3.1
 \subsection{Regular normal cone}

  For the set ${\rm gph}\,\partial\|\cdot\|_*$, we shall verify that its regular normal cone
  coincides with its proximal normal cone just as \cite{WZZhang14} did
  for ${\rm gph}\,\mathcal{N}_{\mathbb{S}_{+}^m}(\cdot)$. This requires the following
  two lemmas. Among others, Lemma \ref{tangent-graphf} characterizes the tangent cone
  to ${\rm gph}\,\partial\|\cdot\|_*$, while Lemma \ref{prox-normalcone} provides
  the characterization for the proximal normal cone to ${\rm gph}\,\partial\|\cdot\|_*$.
 %------------------------------------------------------------------------------------------Lemma 3.1
 \begin{lemma}\label{tangent-graphf}
  For any given $(X,Y)\!\in{\rm gph}\,\partial\|\cdot\|_*$, the following equalities hold:
  \[
    \mathcal{T}_{{\rm gph}\,\partial\|\cdot\|_*}^{i}(X,Y)\!=\!\mathcal{T}_{{\rm gph}\,\partial\|\cdot\|_*}(X,Y)
    =\!\big\{(G,H)\in\mathbb{R}^{m\times n}\times\mathbb{R}^{m\times n}\ |\
    \Pi_{\mathbb{B}}'(X\!+\!Y,G\!+\!H)=\!H\big\}.
  \]
 \end{lemma}
 \begin{proof}
  Let $(G,H)$ be an arbitrary point from $\mathcal{T}_{{\rm gph}\,\partial\|\cdot\|_*}(X,Y)$.
  By the definition of tangent cone, there exist $t_k\downarrow 0$ and $(G^k,H^k)\to(G,H)$ such that
  $(X,Y)+t_k(G^k,H^k)\in {\rm gph}\,\partial\|\cdot\|_*$ for each $k$. By Lemma \ref{graph-subdiff},
  $\Pi_{\mathbb{B}}(X+Y+t_k(G^k+H^k))=Y+t_kH^k$. Notice that $Y=\Pi_{\mathbb{B}}(X+Y)$
  by virtue of $(X,Y)\in{\rm gph}\,\partial\|\cdot\|_*$ and Lemma \ref{graph-subdiff}.
  Then we have
  \[
    \Pi_{\mathbb{B}}'(X\!+\!Y,G\!+\!H)=
    \lim_{k\to\infty}\frac{1}{t_k}\big(\Pi_{\mathbb{B}}(X\!+\!Y+t_k(G^k\!+\!H^k))-\Pi_{\mathbb{B}}(X\!+\!Y)\big)
    =H.
  \]
  This, by the arbitrariness of $(G,H)$ in $\mathcal{T}_{{\rm gph}\,\partial\|\cdot\|_*}(X,Y)$,
  implies the following inclusion:
  \[
    \mathcal{T}_{{\rm gph}\,\partial\|\cdot\|_*}(X,Y)\subseteq
    \big\{(G,H)\in\mathbb{R}^{m\times n}\times\mathbb{R}^{m\times n}\ |\
    \Pi_{\mathbb{B}}'(X\!+\!Y,G\!+\!H)=H\big\}.
  \]
  Since $\mathcal{T}_{{\rm gph}\,\partial\|\cdot\|_*}^{i}(X,Y)\!\subseteq\!\mathcal{T}_{{\rm gph}\,\partial\|\cdot\|_*}(X,Y)$,
  the rest only needs to establish the inclusion
  \begin{equation}\label{converse-inclusion}
    \big\{(G,H)\in\mathbb{R}^{m\times n}\times\mathbb{R}^{m\times n}\ |\ \Pi_{\mathbb{B}}'(X+Y,G+H)=H\big\}
    \subseteq\mathcal{T}_{{\rm gph}\,\partial\|\cdot\|_*}^{i}(X,Y).
  \end{equation}
  To this end, let $(G,H)\in\mathbb{R}^{m\times n}\times\mathbb{R}^{m\times n}$
  with $\Pi_{\mathbb{B}}'(X\!+\!Y,G\!+\!H)=H$. For any $t>0$, write $Z_t:=X\!+\!Y+t(G\!+\!H)$.
  By the definition of $\Pi_{\mathbb{B}}(\cdot)$, we have
  \(
    Z_t-\Pi_{\mathbb{B}}(Z_t)\in\mathcal{N}_{\mathbb{B}}(\Pi_{\mathbb{B}}(Z_t)).
  \)
  In addition, from $\Pi_{\mathbb{B}}'(X+Y,G+H)=H$ and the definition of the directional derivative,
  \[
    \Pi_{\mathbb{B}}(Z_t)=\Pi_{\mathbb{B}}(X+Y)+tH+o(t)=Y+tH+o(t).
  \]
  This shows that $X+tG+o(t)\in\mathcal{N}_{\mathbb{B}}(\Pi_{\mathbb{B}}(Z_t))$,
  and then $\big(X+tG+o(t),\Pi_{\mathbb{B}}(Z_t)\big)\in{\rm gph}\,\partial\|\cdot\|_*$ by Lemma \ref{graph-subdiff}.
  Along with the last equality,
  \(
    {\rm dist}\big((X+tG,Y+tH),{\rm gph}\,\partial\|\cdot\|_*\big)
    =o(t).
  \)
  This means that $(G,H)\in\mathcal{T}_{{\rm gph}\,\partial\|\cdot\|_*}^{i}(X,Y)$.
  So, the inclusion in \eqref{converse-inclusion} follows.
 \end{proof}
%------------------------------------------------------------------------------------------------Lemma 3.2
 \begin{lemma}\label{prox-normalcone}
  For any given $(X,Y)\!\in\!{\rm gph}\,\partial\|\cdot\|_*$, we have
  $(X^*,Y^*)\!\in\!\mathcal{N}_{{\rm gph}\,\partial\|\cdot\|_*}^{\pi}\!(X,Y)$ iff
  \begin{equation}\label{aim-ineq}
    \langle X^*, W-\Pi'_{\mathbb{B}}(X+Y;W)\rangle+\langle Y^*,\Pi'_{\mathbb{B}}(X+Y;W)\rangle\le 0
    \quad\ \forall W\in\mathbb{R}^{m\times n}.
  \end{equation}
 \end{lemma}
 \begin{proof}
  Let $(X^*,Y^*)\in\mathcal{N}_{{\rm gph}\,\partial\|\cdot\|_*}^{\pi}(X,Y)$.
  We prove that inequality \eqref{aim-ineq} holds. For this purpose,
  let $W$ be an arbitrary point from $\mathbb{R}^{m\times n}$. For any $t>0$, we write
  \begin{equation}\label{temp-ineq0}
   Y'_t:=\Pi_{\mathbb{B}}(X+Y+tW)\ \ {\rm and}\ \  X'_t:=X+Y+tW-\Pi_{\mathbb{B}}(X+Y+tW).
  \end{equation}
  Clearly, $X_t'\in\mathcal{N}_{\mathbb{B}}(Y_t')$. By Lemma \ref{graph-subdiff},
  $Y'_t\in\partial\|X'_t\|_*$.
  Since $(X^*,Y^*)\in\mathcal{N}_{{\rm gph}\,\partial\|\cdot\|_*}^{\pi}(X,Y)$,
  by Part E of \cite[Chapter 6]{RW98} there exists $\varepsilon>0$ such that
  for any $(X',Y')\in {\rm gph}\,\partial\|\cdot\|_*$,
 \begin{equation*}
  \langle (X^*, Y^*),(X',Y')-(X,Y)\rangle\leq \varepsilon\|(X',Y')-(X,Y)\|_F^2.
 \end{equation*}
  Take $(X',Y')=(X'_t,Y_t')$. From this inequality, it follows that
 \begin{equation}\label{temp-ineq}
  \langle (X^*, Y^*),(X_t'-X,Y_t'-Y)\rangle\leq \varepsilon\|(X_t',Y_t')-(X,Y)\|_F^2.
 \end{equation}
  Note that $Y=\Pi_{\mathbb{B}}(X+Y)$ since $(X,Y)\in{\rm gph}\,\partial\|\cdot\|_*$.
  From \eqref{temp-ineq0} and \eqref{temp-ineq}, we have that
 \begin{align*}
  &\langle X^*, W-\Pi'_{\mathbb{B}}(X+Y;W)\rangle+\langle Y^*,\Pi'_{\mathbb{B}}(X+Y;W)\rangle \\
  &\le \varepsilon\lim_{t\downarrow 0}\frac{1}{t}(\| X'_t-X\|_F^2+\|Y'_t-Y\|_F^2)\nonumber \\
  &\leq \varepsilon \lim_{t\downarrow 0}\frac{1}{t}\big(3\|Y-\!\Pi_{\mathbb{B}}(X+Y+tW)\|_F^2+2t^2\|W\|_F^2\big)
  \le\varepsilon\lim_{t\downarrow 0}\frac{1}{t}\big(5t^2\|W\|_F^2\big)=0,
 \end{align*}
 where the last inequality is using $Y=\Pi_{\mathbb{B}}(X+Y)$ and the global Lipschitz continuity
 with modulus $1$ of the projection operator $\Pi_{\mathbb{B}}(\cdot)$.
 This shows that \eqref{aim-ineq} holds. Conversely, suppose that \eqref{aim-ineq} holds.
 We shall prove $(X^*,Y^*)\in\mathcal{N}_{{\rm gph}\,\partial\|\cdot\|_*}^{\pi}(X,Y)$.
 By Lemma \ref{Calmly-Bdiff}, there exist $\delta>0$ and a constant $M>0$ such that
 for any $Z'$ with $\|Z'-(X+Y)\|_F\leq \delta$,
 \begin{align*}
 \langle X^*,\Pi_{\mathbb{B}}(Z')-\Pi_{\mathbb{B}}(X+Y)\rangle
  \le \langle Y^*,\Pi_{\mathbb{B}}'(X+Y;Z'-X-Y)\rangle+M\|Z'-(X+Y)\|_F^2,\\
  \langle Y^*,\Pi_{\mathbb{B}}(Z')-\Pi_{\mathbb{B}}(X+Y)\rangle
  \le \langle Y^*,\Pi_{\mathbb{B}}'(X+Y;Z'-X-Y)\rangle+M\|Z'-(X+Y)\|_F^2.
 \end{align*}
 Thus, for any $(X',Y')\in{\rm gph}\,\partial\|\cdot\|_*$ with $\|(X',Y')-(X,Y)\|_F\le \delta/2$, we have that
 \begin{align}\label{equa1-XYStar}
   &\langle X^*,\Pi_{\mathbb{B}}(X'\!+\!Y')-\Pi_{\mathbb{B}}(X\!+\!Y)\rangle
     -\langle X^*,\Pi_{\mathbb{B}}'(X\!+\!Y;\Delta X\!+\!\Delta Y)\rangle\le M\|\Delta X\!+\!\Delta Y\|_F^2,\\
   &\langle Y^*,\Pi_{\mathbb{B}}(X'\!+\!Y')-\Pi_{\mathbb{B}}(X\!+\!Y)\rangle
     -\langle Y^*,\Pi_{\mathbb{B}}'(X\!+\!Y;\Delta X\!+\!\Delta Y)\rangle
   \le M\|\Delta X \!+\!\Delta Y\|_F^2,
   \label{equa2-XYStar}
 \end{align}
 where $\Delta X\!=\!X'-X$ and $\Delta Y=Y'-Y$. Along with $\Pi_{\mathbb{B}}(X'+Y')=Y'$
 and $\Pi_{\mathbb{B}}(X+Y)=Y$,
 \begin{align*}
  \langle (X^*, Y^*),(X',Y')-(X,Y)\rangle
  &=\langle X^*,\Delta X\!+\!\Delta Y-\Pi_{\mathbb{B}}(X'+Y')+\Pi_{\mathbb{B}}(X+Y)\rangle\\
  &\quad +\langle Y^*,\Pi_{\mathbb{B}}(X'+Y')-\Pi_{\mathbb{B}}(X+Y)\rangle\\
  &\le\langle X^*, \Delta X\!+\!\Delta Y-\Pi'_{\mathbb{B}}(X+Y;\Delta X\!+\!\Delta Y)\rangle\nonumber \\
  &\qquad+ \langle Y^*, \Pi'_{\mathbb{B}}(X+Y; \Delta X\!+\!\Delta Y)\rangle + 2M\|\Delta X\!+\!\Delta Y\|_F^2\nonumber\\
  &\leq 4M\|\Delta X\!+\!\Delta Y\|_F^2\nonumber,
 \end{align*}
 where the first inequality is using \eqref{equa1-XYStar} and \eqref{equa2-XYStar},
 and the last one is by virtue of \eqref{aim-ineq} with $W=\Delta X\!+\!\Delta Y$.
 Take $\varepsilon=\max\{4M, 2\|(X^*,Y^*)\|_F/\delta\}$. For any $(X',Y')\in {\rm gph}\,\partial\|\cdot\|_*$,
 it holds that
 \(
  \langle (X^*, Y^*),(X',Y')-(X,Y)\rangle\leq \varepsilon\|(X',Y')-(X,Y)\|_F^2.
 \)
 This, by Part E of \cite[Chapter 6]{RW98}, shows that
 $(X^*,Y^*)\in\mathcal{N}_{{\rm gph}\,\partial\|\cdot\|_*}^{\pi}(X,Y)$.
 Thus, we finish the proof.
 \end{proof}

  Now we are in a position to establish the coincidence between the regular normal cone
  to ${\rm gph}\,\partial\|\cdot\|_*$ and the proximal normal cone to ${\rm gph}\,\partial\|\cdot\|_*$.
 %-------------------------------------------------------------------------------------------Corollary
 \begin{proposition}\label{prop-relation}
  For any given $(X,Y)\!\in{\rm gph}\,\partial\|\cdot\|_*$,
  \(
   \mathcal{\widehat{N}}_{{\rm gph}\,\partial\|\cdot\|_*}(X,Y)\!=\!\mathcal{N}^{\pi}_{{\rm gph}\,\partial\|\cdot\|_*}(X,Y).
  \)
 \end{proposition}
 \vspace{-0.3cm}
 \begin{proof}
  Take an arbitrary point $(X^*,Y^*)\in\widehat{\mathcal{N}}_{{\rm gph}\,\partial\|\cdot\|_*}(X,Y)$.
  By \cite[Proposition 6.5]{RW98}, $\langle (X^*,Y^*),(G,H)\rangle\le 0$
  for any $(G,H)\in \mathcal{T}_{{\rm gph}\,\partial\|\cdot\|_*}(X,Y)$.
  From Lemma \ref{tangent-graphf}, clearly,
  \(
   \big(W-\Pi_{\mathbb{B}}'(X+Y,W),\Pi_{\mathbb{B}}'(X+Y,W)\big)\in \mathcal{T}_{{\rm gph}\,\partial\|\cdot\|_*}(X,Y)
  \)
  for any $W\in\mathbb{R}^{m\times n}$, and then
  \(
   \langle X^*, W-\Pi_{\mathbb{B}}'(X+Y,W)\rangle+\langle Y^*, \Pi_{\mathbb{B}}'(X+Y,W)\rangle\leq 0.
  \)
  This, by Lemma \ref{prox-normalcone}, implies that $(X^*,Y^*)\in \mathcal{N}^{\pi}_{{\rm gph}\,\partial\|\cdot\|_*}(X,Y)$,
  and then $\widehat{\mathcal{N}}_{{\rm gph}\,\partial\|\cdot\|_*}(X,Y)\subseteq\mathcal{N}^{\pi}_{{\rm gph}\,\partial\|\cdot\|_*}(X,Y)$
  follows. Next take an arbitrary point $(X^*,Y^*)\in\mathcal{N}^{\pi}_{{\rm gph}\,\partial\|\cdot\|_*}(X,Y)$.
  For any $(G,H)\in\mathcal{T}_{{\rm gph}\,\partial\|\cdot\|_*}(X,Y)$,
  by Lemma \ref{tangent-graphf} it follows that $\Pi_{\mathbb{B}}'(X+Y,G+H)=H$.
  Using Lemma \ref{prox-normalcone} with $W=G+H$ and noting that $\Pi_{\mathbb{B}}'(X+Y,W)=H$
  yields that
  \(
    \langle (X^*,Y^*),(G,H)\rangle\le 0,
  \)
  i.e., $(X^*,Y^*)\in \mathcal{\widehat{N}}_{{\rm gph}\,\partial\|\cdot\|_*}(X,Y)$.
  So, $\mathcal{N}^{\pi}_{{\rm gph}\,\partial\|\cdot\|_*}(X,Y)\subseteq
  \widehat{\mathcal{N}}_{{\rm gph}\,\partial\|\cdot\|_*}(X,Y)$.
  The proof is completed.
 \end{proof}

 \medskip

 Proposition \ref{prop-relation} shows that, to characterize the regular normal cone to
 ${\rm gph}\,\partial\|\cdot\|_*$, one only needs to characterize its proximal normal cone.
 Next we shall employ Lemma \ref{prox-normalcone} and Lemma \ref{dir-derivative} to
 derive the expression of the proximal normal cone to ${\rm gph}\,\partial\|\cdot\|_*$.
%-------------------------------------------------------------------------------------------------
 \begin{theorem}\label{proxnormal-theorem}
  For any given $(X,Y)\in{\rm gph}\,\partial\|\cdot\|_*$, let $\overline{Z}=X+Y$ have the SVD as
  given in Lemma \ref{dir-derivative}. With $\Omega_1$ and $\Omega_2$ in \eqref{Omega1}-\eqref{Omega2},
  we define the following matrices
  \begin{subequations}
  \begin{equation*}\label{Theta}
  \Theta_1:=\left[\begin{matrix}
            0_{\alpha\alpha} & 0_{\alpha \beta} &(\Omega_1)_{\alpha\gamma}\\
            0_{\beta\alpha}& 0_{\beta\beta}&E_{\beta\gamma}\\
            (\Omega_1)_{\gamma\alpha}& E_{\gamma\beta}&E_{\gamma\gamma}\\
           \end{matrix}\right],\
  \Theta_2:=\left[\begin{matrix}
             E_{\alpha\alpha} & E_{\alpha \beta} &E_{\alpha\gamma}\!-\!(\Omega_1)_{\alpha\gamma}\\
             E_{\beta\alpha}& 0_{\beta\beta}&0_{\beta\gamma}\\
             E_{\gamma\alpha}\!-\!(\Omega_1)_{\gamma\alpha}& 0_{\gamma\beta}&0_{\gamma\gamma}\\
            \end{matrix}\right],\\
  \end{equation*}
 \begin{equation*}\label{Sigma}
  \Sigma_1\!:=\!\left[\begin{matrix}
             (\Omega_2)_{\alpha\alpha} & (\Omega_2)_{\alpha\beta} & (\Omega_2)_{\alpha\gamma}\\
             (\Omega_2)_{\beta\alpha}& 0_{\beta\beta}&E_{\beta\gamma}\\
             (\Omega_2)_{\gamma\alpha}& E_{\gamma \beta}&E_{\gamma\gamma}\\
             \end{matrix}\right],
  \Sigma_2\!:=\!\left[\begin{matrix}
              E_{\alpha\alpha}\!-\!(\Omega_2)_{\alpha\alpha} & E_{\alpha \beta}\!-\!(\Omega_2)_{\alpha\beta}
              &E_{\alpha\gamma}\!-\!(\Omega_2)_{\alpha\gamma}\\
              E_{\beta\alpha}\!-\!(\Omega_2)_{\beta\alpha }& 0_{\beta \beta}&0_{\beta\gamma}\\
              E_{\gamma\alpha}\!-\!(\Omega_2)_{\gamma\alpha}& 0_{\gamma\beta}&0_{\gamma\gamma}\\
             \end{matrix}\!\right].
  \end{equation*}
  \end{subequations}
  Then $(X^*,Y^*)\in \mathcal{N}_{{\rm gph}\,\partial\|\cdot\|_*}^{\pi}(X,Y)$
  if and only if $(X^*,Y^*)$ satisfies the following conditions
  \vspace{-0.3cm}
  \begin{subequations}
  \begin{align}\label{result-equa1}
   \Theta_{1}\circ \mathcal{S}(\widetilde{Y}_1^*)+ \Theta_{2}\circ \mathcal{S}(\widetilde{X}_1^*)
    +\Sigma_{1}\circ \mathcal{X}(\widetilde{Y}_1^*)+\Sigma_{2}\circ \mathcal{X}(\widetilde{X}_1^*)=0,\\
    \label{result-equa2}
    \widetilde{X}_{\alpha c}^*\circ (E_{\alpha c}-(\Omega_3)_{\alpha c})+\widetilde{Y}_{\alpha c}^*\circ (\Omega_3)_{\alpha c}=0,\qquad\qquad\\
    \widetilde{Y}_{\beta c}^*=0,\, \widetilde{Y}_{\gamma c}^*=0,\,\widetilde{X}_{\beta\beta}^*\preceq 0,\,
    \widetilde{Y}_{\beta\beta}^*\succeq 0,\qquad\qquad\quad
    \label{result-equa3}
  \end{align}
  \end{subequations}
  where $\widetilde{X}_1^*=\overline{U}^{\mathbb{T}}X^*\overline{V}_1,$
  $\widetilde{Y}_1^*=\overline{U}^{\mathbb{T}}Y^*\overline{V}_1$,
  $\widetilde{X}^*=\overline{U}^{\mathbb{T}}X^*\overline{V}$ and
  $\widetilde{Y}^*=\overline{U}^{\mathbb{T}}Y^*\overline{V}$.
 \end{theorem}
 \begin{proof}
  By Lemma \ref{prox-normalcone} and Lemma \ref{dir-derivative},
  $(X^*,Y^*)\in\!\mathcal{N}_{{\rm gph}\,\partial\|\cdot\|_*}^{\pi}(X,Y)$
  iff for any $H\in\mathbb{R}^{m\times n}$,
 \begin{align}\label{main-ineq1}
  &\left\langle \widetilde{X}^*,\widetilde{H}-\left[\begin{matrix}
        (\Omega_2)_{\alpha\alpha}\circ \big(\mathcal{X}(\widetilde{H}_1)\big)_{\alpha\alpha}
        &(\Omega_2)_{\alpha\beta}\!\circ\big(\mathcal{X}(\widetilde{H}_1)\big)_{\alpha\beta}
        & \widehat{H}_{\alpha\gamma}&(\Omega_3)_{\alpha c}\!\circ\!\widetilde{H}_{\alpha c}\\
        (\Omega_2)_{\beta\alpha }\!\circ\! \big(\mathcal{X}(\widetilde{H}_1)\big)_{\beta\alpha}&\widetilde{H}_{\beta\beta} \!-\!\Pi_{\mathbb{S}_{+}^{|\beta|}}(\mathcal{S}(\widetilde{H}_{\beta\beta}))&\widetilde{H}_{\beta\gamma}&\widetilde{H}_{\beta c}\\
        \widehat{H}_{\gamma\alpha}&\widetilde{H}_{\gamma\beta}&\widetilde{H}_{\gamma\gamma}&\widetilde{H}_{\gamma c}\\
      \end{matrix}\right]\right\rangle\nonumber\\
  &+\left\langle \widetilde{Y}^*,\left[\begin{matrix}
        (\Omega_2)_{\alpha\alpha}\circ \big(\mathcal{X}(\widetilde{H}_1)\big)_{\alpha\alpha}
        &(\Omega_2)_{\alpha\beta}\!\circ\big(\mathcal{X}(\widetilde{H}_1)\big)_{\alpha\beta}& \widehat{H}_{\alpha\gamma}
        &(\Omega_3)_{\alpha c}\!\circ\!\widetilde{H}_{\alpha c}\\
        (\Omega_2)_{\beta\alpha }\!\circ\! \big(\mathcal{X}(\widetilde{H}_1)\big)_{\beta\alpha }&\widetilde{H}_{\beta\beta} \!-\!\Pi_{\mathbb{S}_{+}^{|\beta|}}(\mathcal{S}(\widetilde{H}_{\beta\beta}))&
          \widetilde{H}_{\beta\gamma}&\widetilde{H}_{\beta c}\\
        \widehat{H}_{\gamma\alpha}&\widetilde{H}_{\gamma \beta}&\widetilde{H}_{\gamma\gamma}&\widetilde{H}_{\gamma c}\\
      \end{matrix}\right]\right\rangle\le 0
 \end{align}
  where $\widetilde{H}_1=\overline{U}^{\mathbb{T}}H\overline{V}_1$ and
  $\widetilde{H}=\overline{U}^{\mathbb{T}}H\overline{V}$.
  Take $H=\overline{U}_{\!\beta}M_{\beta\gamma}\overline{V}_{\!\gamma}^{\mathbb{T}}$
  for any $M_{\beta\gamma}\in\mathbb{R}^{|\beta|\times\gamma}$. By the expressions
  of $\widetilde{H}_1$ and $\widetilde{H}$ and equation \eqref{main-ineq1},
  it is easy to obtain that $\widetilde{Y}_{\beta\gamma}^*=0$. Using the similar arguments, we can achieve that
  $\widetilde{Y}_{\gamma\beta}^*=0,\widetilde{Y}_{\gamma\gamma}^*=0,\widetilde{Y}_{\beta c}^*=0,\widetilde{Y}_{\gamma c}^*=0$ and
  \(
     \widetilde{X}_{\alpha c}^*\circ (E_{\alpha c}\!-\!(\Omega_3)_{\alpha c})+\widetilde{Y}_{\alpha c}^*\circ (\Omega_3)_{\alpha c}=0.
  \)
  Taking $H=\overline{U}_{\!\beta}M_{\beta\beta}\overline{V}_{\!\beta}^{\mathbb{T}}$ for
  any $M_{\beta\beta}\in\mathbb{S}_{-}^{|\beta|}$, from \eqref{main-ineq1} we have
  $\widetilde{Y}_{\beta\beta}^*\succeq 0$; and by taking $H=\overline{U}_{\!\beta}M_{\beta\beta}\overline{V}_{\!\beta}^{\mathbb{T}}$
  for any $M_{\beta\beta}\in\mathbb{S}_{+}^{|\beta|}$, we obtain that
  \begin{align*}
   0\ge \langle \widetilde{X}_{\beta\beta}^*-\widetilde{Y}_{\beta\beta}^*,
   \Pi_{\mathbb{S}_{+}^{|\beta|}}(\mathcal{S}(\widetilde{H}_{\beta\beta}))\rangle
    +\langle \widetilde{Y}_{\beta\beta}^*,\widetilde{H}_{\beta\beta}\rangle
    =\langle \widetilde{X}_{\beta\beta}^*,M_{\beta\beta}\rangle,
  \end{align*}
  which implies that $\widetilde{X}_{\beta\beta}^*\preceq 0$. In addition,
  taking $H=\overline{U}_{\!\alpha}M_{\alpha\alpha}\overline{V}_{\!\alpha}^{\mathbb{T}}$ for
  any $M_{\alpha\alpha}\in\mathbb{R}^{|\alpha|\times|\alpha|}$ and observing that
  $\langle Z,\mathcal{X}(M_{\alpha\alpha})\rangle=\langle\mathcal{X}(Z),M_{\alpha\alpha}\rangle$ for
  any $Z\in\mathbb{R}^{|\alpha|\times|\alpha|}$, from \eqref{main-ineq1} we have
  \begin{align*}
    0\ge \langle\widetilde{X}_{\alpha\alpha}^*,M_{\alpha\alpha}\rangle
    +\langle (\widetilde{Y}_{\alpha\alpha}^*\!-\!\widetilde{X}_{\alpha\alpha}^*)\circ(\Omega_2)_{\alpha\alpha},\mathcal{X}(M_{\alpha\alpha})\rangle\\
    =\langle\widetilde{X}_{\alpha\alpha}^*,M_{\alpha\alpha}\rangle
    +\langle \mathcal{X}\big(\widetilde{Y}_{\alpha\alpha}^*\!-\!\widetilde{X}_{\alpha\alpha}^*\big)\circ(\Omega_2)_{\alpha\alpha},M_{\alpha\alpha}\rangle,
  \end{align*}
  which implies that
  \(
    \widetilde{X}_{\alpha\alpha}^*+\mathcal{X}\big(\widetilde{Y}_{\alpha\alpha}^*\!-\!\widetilde{X}_{\alpha\alpha}^*\big)\circ(\Omega_2)_{\alpha\alpha}=0.
  \)
  Similarly, taking $H=\overline{U}_{\!\alpha}M_{\alpha\beta}\overline{V}_{\!\beta}^{\mathbb{T}}$ for
  any $M_{\alpha\beta}\in\mathbb{R}^{|\alpha|\times|\beta|}$, from equation \eqref{main-ineq1} we obtain that
  \begin{align*}
   0&\ge \langle \widetilde{X}_{\alpha\beta}^*,M_{\alpha\beta}\rangle
        +\langle \big(\mathcal{X}(\widetilde{Y}_1^*\!-\!\widetilde{X}_1^*)\big)_{\alpha\beta}\circ (\Omega_2)_{\alpha\beta},M_{\alpha\beta}\rangle,
  \end{align*}
  which shows that $\widetilde{X}_{\alpha\beta}^*+\big(\mathcal{X}(\widetilde{Y}_1^*\!-\!\widetilde{X}_1^*)\big)_{\alpha\beta}
  \circ (\Omega_2)_{\alpha\beta}=0$. Using the similar way, we have
  \begin{subequations}
  \begin{align}
    \label{equa1-ab}
    \widetilde{X}_{\beta\alpha}^*+\big(\mathcal{X}(\widetilde{Y}_1^*\!-\!\widetilde{X}_1^*)\big)_{\beta\alpha}
    \circ (\Omega_2)_{\beta\alpha}=0,\qquad\qquad\\
    \label{equa1-ag}
    \widetilde{X}_{\alpha\gamma}^*+\big(\mathcal{S}(\widetilde{Y}_1^*\!-\!\widetilde{X}_1^*)\big)_{\alpha\gamma}
   \circ(\Omega_1)_{\alpha\gamma}+\big(\mathcal{X}(\widetilde{Y}_1^*\!-\!\widetilde{X}_1^*)\big)_{\alpha\gamma}
   \circ(\Omega_2)_{\alpha\gamma}=0,\\
      \label{equa2-ag}
   \widetilde{X}_{\gamma\alpha}^*+\big(\mathcal{S}(\widetilde{Y}_1^*\!-\!\widetilde{X}_1^*)\big)_{\gamma\alpha}
   \circ(\Omega_1)_{\gamma\alpha}+\big(\mathcal{X}(\widetilde{Y}_1^*\!-\!\widetilde{X}_1^*)\big)_{\gamma\alpha}
   \circ(\Omega_2)_{\gamma\alpha}=0.
  \end{align}
 \end{subequations}
  To sum up, the fact that inequality \eqref{main-ineq1} holds for any $H\in\mathbb{R}^{m\times n}$ implies that
  \begin{align}\label{main-equa2}
  \widetilde{X}_{\alpha\alpha}^*+\mathcal{X}(\widetilde{Y}_{\alpha\alpha}^*\!-\!\widetilde{X}_{\alpha\alpha}^*)\circ(\Omega_2)_{\alpha\alpha}=0,\ \ \widetilde{X}_{\beta\beta}^*\preceq 0,\ \widetilde{Y}_{\beta\beta}^*\succeq 0,\qquad\nonumber\\
  \widetilde{X}_{\alpha\beta}^*+\big(\mathcal{X}(\widetilde{Y}_1^*\!-\!\widetilde{X}_1^*)\big)_{\alpha\beta}\circ (\Omega_2)_{\alpha\beta}=0,\,
   \widetilde{X}_{\beta\alpha}^*+\big(\mathcal{X}(\widetilde{Y}_1^*\!-\!\widetilde{X}_1^*)\big)_{\beta\alpha}\circ (\Omega_2)_{\alpha\beta}=0,\nonumber\\
    \widetilde{X}_{\alpha\gamma}^*+\big(\mathcal{S}(\widetilde{Y}_1^*\!-\!\widetilde{X}_1^*)\big)_{\alpha\gamma}
   \circ(\Omega_1)_{\alpha\gamma}+\big(\mathcal{X}(\widetilde{Y}_1^*\!-\!\widetilde{X}_1^*)\big)_{\alpha\gamma}
   \circ(\Omega_2)_{\alpha\gamma}=0,\qquad\\
   \widetilde{X}_{\gamma\alpha}^*+\big(\mathcal{S}(\widetilde{Y}_1^*\!-\!\widetilde{X}_1^*)\big)_{\gamma\alpha}
   \circ(\Omega_1)_{\gamma\alpha}+\big(\mathcal{X}(\widetilde{Y}_1^*\!-\!\widetilde{X}_1^*)\big)_{\gamma\alpha}
   \circ(\Omega_2)_{\gamma\alpha}=0,\,\widetilde{Y}_{\beta\gamma}^*=0,\qquad\nonumber\\
   \widetilde{Y}_{\gamma\beta}^*=0,\widetilde{Y}_{\gamma\gamma}^*=0,\widetilde{Y}_{\beta c}^*=0,\widetilde{Y}_{\gamma c}^*=0,\,
   \widetilde{X}_{\alpha c}^*\circ (E_{\alpha c}-(\Omega_3)_{\alpha c})+\widetilde{Y}_{\alpha c}^*\circ (\Omega_3)_{\alpha c}=0.\nonumber
  \end{align}
  By the definitions of $\Theta_1,\Theta_2$ and $\Sigma_1,\Sigma_2$, equation
  \eqref{main-equa2} can be compactly written as \eqref{result-equa1}-\eqref{result-equa3}.
  Conversely, it is easy to check that if $(\widetilde{X}^*,\widetilde{Y}^*)$ satisfies
  \eqref{main-equa2} or its compact form \eqref{result-equa1}-\eqref{result-equa3},
  then \eqref{main-ineq1} holds for any $H\in\mathbb{R}^{m\times n}$, i.e.,
  $(X^*,Y^*)\in\!\mathcal{N}_{{\rm gph}\,\partial\|\cdot\|_*}^{\pi}(X,Y)$.
  \end{proof}
%-------------------------------------------------------------------------------------------------
 \begin{remark}\label{remark-pnormal}
   For any given $(X,Y)\in{\rm gph}\,\partial\|\cdot\|_*$, let $\overline{Z}=X+Y$.
   By Theorem \ref{proxnormal-theorem}, if $\|\overline{Z}\|<1$,
   then $(X^*,Y^*)\in\mathcal{N}_{{\rm gph}\,\partial\|\cdot\|_*}^{\pi}(X,Y)$ if and only if
   $Y^*=0$; if $\|\overline{Z}\|=1$, then $(X^*,Y^*)\in\mathcal{N}_{{\rm gph}\,\partial\|\cdot\|_*}^{\pi}(X,Y)$
   if and only if $\widetilde{X}^*$ and $\widetilde{Y}^*$ take the following form
   \[
    \widetilde{X}^*\!=\!\left[\begin{matrix}
                   \widetilde{X}_{\beta\beta}^*& \widetilde{X}_{\beta\gamma}^*&\widetilde{X}_{\beta c}^*\\
                    \widetilde{X}_{\gamma\beta}^*& \widetilde{X}_{\gamma\gamma}^*&\widetilde{X}_{\gamma c}^*\\
                    \end{matrix}\right]\ \ {\rm and}\ \
    \widetilde{Y}^*\!=\!\left[\begin{matrix}
                   \widetilde{Y}_{\beta\beta}^*& 0_{\beta\gamma}&0_{\beta c}\\
                   0_{\gamma\beta}& 0_{\gamma\gamma}& 0_{\gamma c}\\
                    \end{matrix}\right]\ {\rm with}\ \widetilde{X}_{\beta\beta}^*\preceq 0,\,
                    \widetilde{Y}_{\beta\beta}^*\succeq 0.
   \]
 \end{remark}

%-----------------------------------------------------------------------------------------Limiting Norm
 \subsection{Limiting normal cone}\label{subsec3.2}

 Let $\beta$ be a nonempty index set and denote the set of all partitions of $\beta$
 by $\mathscr{P}(\beta)$. Write
 \(
   \mathbb{R}_{>}^{|\beta|}:=\big\{z\in\mathbb{R}^{|\beta|}\!:\ z_1\ge\cdots\ge z_{|\beta|}>0\big\}.
 \)
 For any $z\in\mathbb{R}_{>}^{|\beta|}$, let $D(z)\in\mathbb{S}^{|\beta|}$ denote
 the generalized first divided difference matrix of $h(t)=\min(1,t)$ at $z$,
 which is defined as
 \begin{equation}\label{Ddiff-matrix}
   (D(z))_{ij}:=\left\{\begin{array}{cl}
               \!\frac{\min(1,z_i)-\min(1,z_j)}{z_i-z_j}\in[0,1]&{\rm if}\ z_i\ne z_j,\\
                 0 &{\rm if}\ z_i=z_j\ge 1,\\
                1 & {\rm otherwise}.
                \end{array}\right.
 \end{equation}
 Write
 \(
   \mathcal{U}_{|\beta|}
   :=\big\{\overline{\Omega}\in\mathbb{S}^{|\beta|}\!:\ \overline{\Omega}=\lim_{k\to\infty}D(z^k),\,
   z^k\to e_{|\beta|},\, z^k\in\mathbb{R}_{>}^{|\beta|}\big\}.
 \)
 For each $\Xi_1\in\mathcal{U}_{|\beta|}$, by equation \eqref{Ddiff-matrix} there exists a partition $(\beta_{+},\beta_{0},\beta_{-})\in\mathscr{P}(\beta)$ such that
  \begin{equation}\label{Xi1-matrix}
   \Xi_1=\left[\begin{matrix}
          0_{\beta_{+}\beta_{+}}& 0_{\beta_{+}\beta_{0}}&(\Xi_1)_{\beta_{+}\beta_{-}}\\
          0_{\beta_{0}\beta_{+}}& 0_{\beta_{0}\beta_{0}}& E_{\beta_{0}\beta_{-}}\\
          (\Xi_1)_{\beta_{+}\beta_{-}}^{\mathbb{T}}&E_{\beta_{-}\beta_{0}}&E_{\beta_{-}\beta_{-}}
         \end{matrix}\right],
 \end{equation}
  where each entry of $(\Xi_1)_{\beta_{+}\beta_{-}}$ belongs to $[0,1]$.
  Let $\Xi_2$ be the matrix associated to $\Xi_1$:
 \begin{equation}\label{Xi2-matrix}
   \Xi_2=\left[\begin{matrix}
          E_{\beta_{+}\beta_{+}}& E_{\beta_{+}\beta_{0}}& E_{\beta_{+}\beta_{-}}\!-\!(\Xi_1)_{\beta_{+}\beta_{-}}\\
          E_{\beta_{0}\beta_{+}}& 0_{\beta_{0}\beta_{0}}& 0_{\beta_{0}\beta_{-}}\\
          E_{\beta_{-}\beta_{+}}\!-\!(\Xi_1)_{\beta_{+}\beta_{-}}^{\mathbb{T}}&0_{\beta_{-}\beta_{0}}&0_{\beta_{-}\beta_{-}}
         \end{matrix}\right].
 \end{equation}
 With the above notations, we shall provide the exact formula for the limiting normal cone to
 ${\rm gph}\,\partial \|\cdot\|_*$ in the following theorem, whose proof is included in Appendix.
%----------------------------------------------------------------------------------------------Theorem
 \begin{theorem}\label{normal-theorem}
  For any given $(X,Y)\in{\rm gph}\,\partial \|\cdot\|_*$, let $\overline{Z}=X+Y$ have
  the SVD as in Lemma \ref{dir-derivative} with $\alpha,\beta,\gamma$ and $c$ defined by
  \eqref{ab-index}-\eqref{gc-index}. Then, $(G,H)\in\mathcal{N}_{{\rm gph}\,\partial \|\cdot\|_*}(X,Y)$
  if and only if $(\widetilde{G},\widetilde{H})$ with $\widetilde{G}=\overline{U}^{\mathbb{T}}G\overline{V}$
  and $\widetilde{H}=\overline{U}^{\mathbb{T}}\!H\overline{V}$ satisfies the following conditions
  \begin{subequations}
  \begin{align}\label{equa1-normal}
  \Theta_1\circ\mathcal{S}(\widetilde{H}_1)+\Theta_2\circ \mathcal{S}(\widetilde{G}_1)
    +\Sigma_1\circ \mathcal{X}(\widetilde{H}_1)+\Sigma_2\circ \mathcal{X}(\widetilde{G}_1)=0,\qquad\qquad\\
    \label{equa2-normal}
   \widetilde{G}_{\alpha c}+(\Omega_3)_{\alpha c}\circ
   (\widetilde{H}_{\alpha c}-\widetilde{G}_{\alpha c})=0,\,\widetilde{H}_{\beta c}=0,\,
   \widetilde{H}_{\gamma c}=0,\qquad\qquad\quad\\
   \label{equa3-normal}
   (\widetilde{G}_{\beta\beta},\widetilde{H}_{\beta\beta})\in\!
   \bigcup_{Q\in\mathbb{O}^{|\beta|}\atop \Xi_1\in\mathcal{U}_{|\beta|}}
   \!\left\{(M,N)\ \bigg|\!\left.\begin{array}{ll}
    \Xi_1\circ\widehat{N}+\Xi_2\circ \mathcal{S}(\widehat{M})+\Xi_2\circ \mathcal{X}(\widehat{N})=0\\
    \quad {\rm with}\ \widehat{N}=Q^{\mathbb{T}}NQ,\,\widehat{M}=Q^{\mathbb{T}}MQ,\\
    \quad Q_{\beta_0}^{\mathbb{T}}MQ_{\beta_0}\preceq 0,\
     Q_{\beta_0}^{\mathbb{T}}NQ_{\beta_0}\succeq 0
   \end{array}\right.\!\right\},
 \end{align}
 \end{subequations}
  where $\widetilde{G}_1=\overline{U}^{\mathbb{T}}G\overline{V}_1,\widetilde{H}_1\!=\overline{U}^{\mathbb{T}}H\overline{V}_1$,
  and $\Theta_1,\Theta_2,\Sigma_1,\Sigma_2$ are the same as those before.
  \end{theorem}

  To close this paper, we point out that Theorem \ref{normal-theorem} also provides
  the characterization for the coderivative of $\Pi_\mathbb{B}$. Indeed,
  by Lemma \ref{graph-subdiff}, ${\rm gph}\,\Pi_{\mathbb{B}}=\mathcal{L}^{-1}({\rm gph}\,\partial\|\cdot\|_*)$ with
  \[
    \mathcal{L}(X,Y):=(X\!-\!Y,Y)\ \ {\rm for}\ \ (X,Y)\in\mathbb{R}^{m\times n}\times\mathbb{R}^{m\times n}.
  \]
  Since the linear map $\mathcal{L}\!:\mathbb{R}^{m\times n}\times\mathbb{R}^{m\times n}\to
  \mathbb{R}^{m\times n}\times\mathbb{R}^{m\times n}$ is onto, from \cite[Exercise 6.7]{RW98}
  \[
   \mathcal{N}_{{\rm gph}\,\Pi_{\mathbb{B}}}(X,Y)\!=\!\mathcal{L}^*(\mathcal{N}_{{\rm gph}\,\partial\|\cdot\|_*}(\mathcal{L}(X,Y))),
  \]
  where $\mathcal{L}^*$ is the adjoint of $\mathcal{L}$.
  By the definition of coderivative (see \cite[Definition 8.33]{RW98}),
  \begin{equation}
    W\in D^*\Pi_\mathbb{B}(X,Y)(S)\Longleftrightarrow (W,W\!-\!S)\in\mathcal{N}_{{\rm gph}\,\partial\|\cdot\|_*}(X\!-\!Y,Y).
  \end{equation}
  Similarly, Theorem \ref{proxnormal-theorem} also provides
  the characterization for the regular coderivative:
  \begin{equation}
    W\in \widehat{D}^*\Pi_\mathbb{B}(X,Y)(S)\Longleftrightarrow (W,W\!-\!S)\in\widehat{\mathcal{N}}_{{\rm gph}\,\partial\|\cdot\|_*}(X\!-\!Y,Y).
  \end{equation}
  In addition, with the help of Theorem \ref{normal-theorem} and the Mordukhovich
  criterion \cite[Proposition 3.5]{Mordu94} on the Aubin property of a multifunction,
  one may easily obtain the practical conditions for the Aubin property of $\partial\|\cdot\|_*$.
  In our future work, we shall use Theorem \ref{proxnormal-theorem}
  and \ref{normal-theorem} to derive the optimality conditions of
  the rank minimization problem \eqref{rank-min}.

 %-------------------------------------------------------------------------Appendix A
  \bigskip
  \noindent
  {\bf\large Appendix}

   \medskip
  \noindent
  {\bf Proof of Theorem \ref{normal-theorem}:}
  Throughout the proof, let $\nu_1\!>\nu_2>\cdots>\nu_r$ denote the nonzero distinct
  singular values of $\overline{Z}$, and $a,b$ and $a_k$ be the index sets defined by
  \begin{subequations}
  \begin{equation}\label{ab}
    a\!:=\!\big\{i\in\{1,\ldots,m\}\ |\ \sigma_i(\overline{Z})>0\big\},\
    b\!:=\!\big\{i\in\{1,\ldots,m\}\ |\ \sigma_i(\overline{Z})=0\big\},
  \end{equation}
  \begin{equation}\label{ak}
    a_k:=\big\{i\in\{1,\ldots,m\}\ |\ \sigma_i(\overline{Z})=\overline{\nu}_k\big\}\quad{\rm for}\ \ k=1,2,\ldots,r.
  \end{equation}
  \end{subequations}
  In addition, we write $\alpha=\bigcup_{i=1}^{l-1}a_i$,
  $\beta=a_l$ and $\gamma_1:=\{i\in\gamma\ |\ \sigma_i(\overline{Z})>0\}=\bigcup_{i=l+1}^{r}a_i$.

  \medskip
  \noindent
  ``$\Longrightarrow$''. Let $(G,H)\in\mathcal{N}_{{\rm gph}\,\partial\|\cdot\|_*}(X,Y)$.
  By Proposition \ref{prop-relation} and the definition of limiting normal cones,
  there exist sequences $(X^k,Y^k)\to (X,Y)$ and $(G^k,H^k) \to (G,H)$ with
  $(G^k,H^k)\in\mathcal{N}_{{\rm gph}\,\partial\|\cdot\|_*}^{\pi}(X^k,Y^k)$ for each $k$.
  For each $k$, we write $Z^k= X^k+Y^k$ and let $Z^k$ have the SVD as
  $U^k{\rm Diag}(\sigma(Z^k))(V^k)^{\mathbb{T}}$ where $U^k\in\mathbb{O}^{m}$
  and $V^k=[V_1^k\ \ V_2^k]\in\mathbb{O}^n$ with $V_1^k\in\mathbb{O}^{n\times m}$.
  Since $\{(U^k,V^k)\}$ is uniformly bounded, by taking a subsequence if necessary,
  we may assume that $\lim_{k\to\infty}(U^k,V^k)=(\widehat{U},\widehat{V})$.
  Clearly, $\overline{Z}=\widehat{U}[{\rm Diag}(\sigma(\overline{Z}))\ \ 0]\widehat{V}^{\mathbb{T}}$.
  By \cite[Proposition 2.5]{DingSST14}, there exist orthogonal matrices $Q'\in\mathbb{O}^{|b|}$
  and $Q''\in\mathbb{O}^{n-|a|}$ and a block diagonal orthogonal matrix $Q={\rm Diag}(Q_1,Q_2,\ldots,Q_r)$
  with $Q_k\in\mathbb{O}^{|a_k|}$ such that
  \begin{equation}\label{widehatUV}
    \widehat{U}=\overline{U}\left[\begin{matrix}
                                Q & 0\\ 0 & Q'
                             \end{matrix}\right]\ \ {\rm and}\ \
    \widehat{V}=\overline{V}\left[\begin{matrix}
                                Q & 0\\ 0 & Q''
                             \end{matrix}\right].
  \end{equation}
  Since $\sigma(\overline{Z})=\lim_{k\to\infty}\sigma(Z^k)$, for all sufficiently large $k$,
  we have $\sigma_i(Z^k)>1$ if $i\in\alpha$ and $\sigma_i(Z^k)<1$ if $i\in\gamma$.
  Since $\lim_{k\to\infty}\sigma_i(Z^k)=1$ for $i\in\beta$,
  we assume (if necessary taking a subsequence) that there exists
  a partition $(\beta_{+},\beta_{0},\beta_{-})$ of $\beta$ such that for each $k$,
  \[
    \sigma_i(Z^k)>1\ \ \forall i\in\!\beta_{+},\ \
    \sigma_i(Z^k)=1\ \ \forall i\in\!\beta_{0}\ \ {\rm and}\ \
    \sigma_i(Z^k)<1\ \ \forall i\in \beta_{-}.
  \]
  Since $(G^k,H^k)\in\mathcal{N}_{{\rm gph}\,\partial\|\cdot\|_*}^{\pi}(X^k,Y^k)$ for each $k$,
  by Theorem \ref{proxnormal-theorem} there exist the matrices
  \begin{align}\label{Theta1k}
   \Theta_1^k&=\left[\begin{matrix}
             0_{\alpha\alpha} & 0_{\alpha\beta_{+}}& 0_{\alpha\beta_{0}}&(\Omega_1^k)_{\alpha\beta_{-}}
             &(\Omega_1^k)_{\alpha\gamma}\\
             0_{\beta_{+}\alpha} & 0_{\beta_{+}\beta_{+}}& 0_{\beta_{+}\beta_{0}}&(\Omega_1^k)_{\beta_{+}\beta_{-}}
             &(\Omega_1^k)_{\beta_{+}\gamma}\\
             0_{\beta_{0}\alpha}& 0_{\beta_{0}\beta_{+}}&0_{\beta_{0}\beta_0}&E_{\beta_{0}\beta_{-}}&E_{\beta_{0}\gamma}\\
            (\Omega_1^k)_{\beta_{-}\alpha}&(\Omega_1^k)_{\beta_{-}\beta_{+}}& E_{\beta_{-}\beta_{0}}&E_{\beta_{-}\beta_{-}}&E_{\beta_{-}\gamma}\\
            (\Omega_1^k)_{\gamma\alpha}&(\Omega_1^k)_{\gamma\beta_{+}}& E_{\gamma\beta_{0}}&E_{\gamma\beta_{-}}&E_{\gamma\gamma}\\
    \end{matrix}\right],\\
   \Theta_2^k&=\!\left[\begin{matrix}
     E_{\alpha\alpha}& E_{\alpha\beta_{+}} & E_{\alpha\beta_{0}} &(\widetilde{\Omega}_1^k)_{\alpha\beta_{-}} &(\widetilde{\Omega}_1^k)_{\alpha\gamma}\\
      E_{\beta_{+}\alpha}& E_{\beta_{+}\beta_{+}} & E_{\beta_{+}\beta_{0}} &(\widetilde{\Omega}_1^k)_{\beta_{+}\beta_{-}} &(\widetilde{\Omega}_1^k)_{\beta_{+}\gamma}\\
      E_{\beta_{0}\alpha} &E_{\beta_{0}\beta_{+}}& 0_{\beta_{0}\beta_{0}}&0_{\beta_{0}\beta_{-}}&0_{\beta_{0}\gamma}\\
     (\widetilde{\Omega}_1^k)_{\beta_{-}\alpha} &(\widetilde{\Omega}_1^k)_{\beta_{-}\beta_{+}}& 0_{\beta_{-}\beta_{0}}
     &0_{\beta_{-}\beta_{-}}&0_{\beta_{-}\gamma}\\
    (\widetilde{\Omega}_1^k)_{\gamma\alpha}& (\widetilde{\Omega}_1^k)_{\gamma\beta_{+}}& 0_{\gamma\beta_{0}}&0_{\gamma\beta_{-}}&0_{\gamma\gamma}\\
   \end{matrix}\right]
   \label{Theta2k}
  \end{align}
   and
   \begin{align}\label{Sigma1k}
   \Sigma_1^k&:=\left[\begin{matrix}
    (\Omega_2^k)_{\alpha\alpha}&(\Omega_2^k)_{\alpha\beta_{+}}& (\Omega_2^k)_{\alpha\beta_{0}}
    & (\Omega_2^k)_{\alpha\beta_{-}} & (\Omega_2^k)_{\alpha\gamma}\\
    (\Omega_2^k)_{\beta_{+}\alpha}&(\Omega_2^k)_{\beta_{+}\beta_{+}}& (\Omega_2^k)_{\beta_{+}\beta_{0}}
    & (\Omega_2^k)_{\beta_{+}\beta_{-}} & (\Omega_2^k)_{\beta_{+}\gamma}\\
   (\Omega_2^k)_{\beta_{0}\alpha}&(\Omega_2^k)_{\beta_{0}\beta_{+}} & 0_{\beta_{0}\beta_{0}}&E_{\beta_{0}\beta_{-}}
    & E_{\beta_{0}\gamma} \\
    (\Omega_2^k)_{\beta_{-}\alpha}&(\Omega_2^k)_{\beta_{-}\beta_{+}}& E_{\beta_{-}\beta_{0}}&E_{\beta_{-}\beta_{-}}
    &E_{\beta_{-}\gamma} \\
   (\Omega_2^k)_{\gamma\alpha}&(\Omega_2^k)_{\gamma\beta_{+}}&E_{\gamma\beta_{0}}&E_{\gamma\beta_{-}}&E_{\gamma\gamma}\\
   \end{matrix}\right],\\
    \Sigma_2^k&:=\!\left[\begin{matrix}
    (\widetilde{\Omega}_2^k)_{\alpha\alpha}& (\widetilde{\Omega}_2^k)_{\alpha\beta_{+}} & (\widetilde{\Omega}_2^k)_{\alpha\beta_{0}} &(\widetilde{\Omega}_2^k)_{\alpha\beta_{-}}&(\widetilde{\Omega}_2^k)_{\alpha\gamma} \\
    (\widetilde{\Omega}_2^k)_{\beta_{+}\alpha}& (\widetilde{\Omega}_2^k)_{\beta_{+}\beta_{+}} & (\widetilde{\Omega}_2^k)_{\beta_{+}\beta_{0}} &(\widetilde{\Omega}_2^k)_{\beta_{+}\beta_{-}}&(\widetilde{\Omega}_2^k)_{\beta_{+}\gamma} \\
    (\widetilde{\Omega}_2^k)_{\beta_{0}\alpha}&(\widetilde{\Omega}_2^k)_{\beta_{0}\beta_{+}}& 0_{\beta_{0}\beta_{0}}
    &0_{\beta_{0}\beta_{-}}&0_{\beta_{0}\gamma}\\
    (\widetilde{\Omega}_2^k)_{\beta_{-}\alpha}&(\widetilde{\Omega}_2^k)_{\beta_{-}\beta_{+}}
    & 0_{\beta_{-}\beta_{0}}&0_{\beta_{-}\beta_{-}}&0_{\beta_{-}\gamma}\\
    (\widetilde{\Omega}_2^k)_{\gamma\alpha}&(\widetilde{\Omega}_2^k)_{\gamma\beta_{+}}&0_{\gamma\beta_{0}}& 0_{\gamma\beta_{-}}& 0_{\gamma\gamma}\\
  \end{matrix}\right]
  \label{Sigma2k}
  \end{align}
  such that
  \begin{subequations}
  \begin{align}\label{WGHk-equa1}
    \Theta_1^k\circ \mathcal{S}(\widetilde{H}_1^k)+\Theta_2^k\circ \mathcal{S}(\widetilde{G}_1^k) +
    \Sigma_1^k\circ \mathcal{X}(\widetilde{H}_1^k)+ \Sigma_2^k\circ \mathcal{X}(\widetilde{G}_1^k)=0,\\
    \label{WGHk-equa2}
     \widetilde{G}^k_{\alpha c}\circ (E_{\alpha c}-(\Omega_3^k)_{\alpha c})
     +\widetilde{H}^k_{\alpha c}\circ (\Omega_3^k)_{\alpha c}=0,\,\widetilde{H}_{\beta_0c}^k=0,
     \,\widetilde{H}_{\beta_{-}c}^k=0,\\
    \label{WGHk-equa3}
    \widetilde{G}^k_{\beta_{+}c}\circ\big[E_{\beta_{+}c}-(\Omega_3^k)_{\beta_{+} c}\big]
    +\widetilde{H}^k_{\beta_{+} c}\circ (\Omega_3^k)_{\beta_{+}c}=0,\,
    \widetilde{H}_{\gamma c}^k=0,\\
   \label{WGHk-equa4}
    \widetilde{G}_{\beta_{0}\beta_{0}}^k\preceq 0,\ \widetilde{H}_{\beta_{0}\beta_{0}}^k\succeq 0\qquad\qquad\qquad\qquad
  \end{align}
  \end{subequations}
  where $\widetilde{\Omega}_1^k\!=\!E-\Omega_1^k$ and $\widetilde{\Omega}_2^k=E-\!\Omega_2^k$
  with $\Omega_1^k,\Omega_2^k$ and $\Omega_3^k$ defined by \eqref{Omega1}-\eqref{Omega3} with $\sigma(Z^k)$,
  $\widetilde{G}_1^k=(U^k)^{\mathbb{T}}G^kV_1^k, \widetilde{H}_1^k=(U^k)^{\mathbb{T}}H^kV_1^k$
  and $\widetilde{G}^k=(U^k)^{\mathbb{T}}G^kV^k,\widetilde{H}^k=(U^k)^{\mathbb{T}}H^kV^k$.
  By the definition of $\Omega_1^k$, we calculate that
  $\lim_{k\to\infty}(\Omega_1^k)_{\alpha\beta_{-}}\!=0_{\alpha\beta_{-}},
  \lim_{k\to\infty}(\Omega_1^k)_{\alpha\gamma}\!=(\Omega_1)_{\alpha\gamma}$ and
  $\lim_{k\to\infty}(\Omega_1^k)_{\beta_{+}\gamma}\!=E_{\beta_{+}\gamma}$.
  Together with the definitions of $\Theta_1^k$ and $\Theta_2^k$,
  there exist $\Xi_1\in\mathcal{U}_{|\beta|}$ and the corresponding $\Xi_2$ defined by
  \eqref{Xi1-matrix} and \eqref{Xi2-matrix}, respectively, such that
  $\lim_{k\to\infty}\Theta_1^k=\widehat{\Theta}_1$ and
  $\lim_{k\to\infty}\Theta_2^k=\widehat{\Theta}_2$ with
  \[
    \widehat{\Theta}_1=\Theta_1+\left[\begin{matrix}
       0_{\alpha\alpha}& 0_{\alpha\beta}&0_{\alpha\gamma}\\
       0_{\beta\alpha} &\Xi_1 & 0_{\beta\gamma} \\
      0_{\gamma\alpha} &  0_{\gamma\beta}
      \end{matrix}\right]\ \ {\rm and}\ \
    \widehat{\Theta}_2=\Theta_2+\left[\begin{matrix}
       0_{\alpha\alpha}& 0_{\alpha\beta}&0_{\alpha\gamma}\\
       0_{\beta\alpha} &\Xi_2 & 0_{\beta\gamma} \\
      0_{\gamma\alpha} &  0_{\gamma\beta} & 0_{\gamma\gamma}\\
     \end{matrix}\right].
  \]
  By the definition of $\Omega_2^k$,
  $\lim_{k\to\infty}(\Omega_2^k)_{\alpha,\alpha\cup\beta\cup\gamma}=(\Omega_2)_{\alpha,\alpha\cup\beta\cup\gamma}$,
  $\lim_{k\to\infty}(\Omega_2^k)_{\beta_{+}\beta_{+}}=E_{\beta_{+}\beta_{+}}$,
  $\lim_{k\to\infty}(\Omega_2^k)_{\beta_{+}\beta_{0}}\!=\!E_{\beta_{+}\beta_{0}}$
  and $\lim_{k\to\infty}(\Omega_2^k)_{\beta_{+},\beta_{-}\cup\gamma}\!=\!E_{\beta_{+},\beta_{-}\cup\gamma}$.
  Then, we have that
   \[
    \lim_{k\to\infty}\Sigma_2^k=\Sigma_2\ \ {\rm and}\ \
    \lim_{k\to\infty}\Sigma_1^k
    =\Sigma_1+
     \!\left[\begin{matrix}
       0_{\alpha\alpha}& 0_{\alpha\beta}&0_{\alpha\gamma}\\
       0_{\beta\alpha} & \Xi_1+\Xi_2 & 0_{\beta\gamma} \\
      0_{\gamma\alpha} &  0_{\gamma\beta} & 0_{\gamma\gamma}\\
     \end{matrix}\right]:=\widehat{\Sigma}_1.
   \]
  Let $\widehat{G}_1=\widehat{U}^{\mathbb{T}}G\widehat{V}_1$
  and $\widehat{H}_1\!=\!\widehat{U}^{\mathbb{T}}H\widehat{V}_1$, where
  $\widehat{V}_1\in\mathbb{O}^{n\times m}$ is the matrix consisting of
  the first $m$ columns of $\widehat{V}$.
  Now taking the limit $k\to\infty$ to equation \eqref{WGHk-equa1} yields that
  \begin{equation}\label{temp-equa1}
    \widehat{\Theta}_1\circ \mathcal{S}(\widehat{H}_1)+\widehat{\Theta}_2\circ \mathcal{S}(\widehat{G}_1)
    +\widehat{\Sigma}_1\circ \mathcal{X}(\widehat{H}_1)+\Sigma_2\circ \mathcal{X}(\widehat{G}_1)=0.
  \end{equation}
  By the definitions of $\widehat{G}_1$ and $\widehat{H}_1$ and equation \eqref{widehatUV},
  one may calculate that
  \begin{align*}
   &\widehat{G}_1
   \!=\!\left[\begin{matrix}
    Q_{\alpha}^{\mathbb{T}}\widetilde{G}_{\alpha\alpha} Q_{\alpha}
   & Q_{\alpha}^{\mathbb{T}}\widetilde{G}_{\alpha\beta} Q_{l}
   & Q_{\alpha}^{\mathbb{T}}\widetilde{G}_{\alpha\gamma_1}Q_{\gamma_1}&
    Q_{\alpha}^{\mathbb{T}}(\widetilde{G}_{\alpha b}Q_{bb}''+\widetilde{G}_{\alpha c}Q_{cb}'')\\
   Q_{l}^{\mathbb{T}}\widetilde{G}_{\beta\alpha} Q_{\alpha}
   & Q_{l}^{\mathbb{T}}\widetilde{G}_{\beta\beta} Q_{l}
   & Q_{l}^{\mathbb{T}}\widetilde{G}_{\beta\gamma_1} Q_{\gamma_1}
   &Q_{l}^{\mathbb{T}}(\widetilde{G}_{\beta b}Q_{bb}''+\widetilde{G}_{\beta c}Q_{cb}'')\\
   Q_{\gamma_1}^{\mathbb{T}}\widetilde{G}_{\gamma_1\alpha} Q_{\alpha}
   & Q_{\gamma_1}^{\mathbb{T}}\widetilde{G}_{\gamma_1\beta} Q_{l}
   & Q_{\gamma_1}^{\mathbb{T}}\widetilde{G}_{\gamma_1\gamma_1} Q_{\gamma_1}
   & Q_{\gamma_1}^{\mathbb{T}}(\widetilde{G}_{\gamma_1b}Q_{bb}''+\widetilde{G}_{\gamma_1c}Q_{cb}'')\\
   (Q')^{\mathbb{T}}\widetilde{G}_{b\alpha} Q_{\alpha}
   & (Q')^{\mathbb{T}}\widetilde{G}_{b\beta} Q_{l}
   & (Q')^{\mathbb{T}}\widetilde{G}_{b\gamma_1} Q_{\gamma_1}
   & (Q')^{\mathbb{T}}(\widetilde{G}_{bb} Q_{bb}''+\widetilde{G}_{bc} Q_{cb}'')\\
   \end{matrix}\right],\\
  &\widehat{H}_1
   \!=\!\left[\begin{matrix}
    Q_{\alpha}^{\mathbb{T}}\widetilde{H}_{\alpha\alpha} Q_{\alpha}
   & Q_{\alpha}^{\mathbb{T}}\widetilde{H}_{\alpha\beta} Q_{l}
   & Q_{\alpha}^{\mathbb{T}}\widetilde{H}_{\alpha\gamma_1}Q_{\gamma_1}&
    Q_{\alpha}^{\mathbb{T}}(\widetilde{H}_{\alpha b}Q_{bb}''+\widetilde{H}_{\alpha c}Q_{cb}'')\\
   Q_{l}^{\mathbb{T}}\widetilde{H}_{\beta\alpha} Q_{\alpha}
   & Q_{l}^{\mathbb{T}}\widetilde{H}_{\beta\beta} Q_{l}
   & Q_{l}^{\mathbb{T}}\widetilde{H}_{\beta\gamma_1} Q_{\gamma_1}
   &Q_{l}^{\mathbb{T}}(\widetilde{H}_{\beta b}Q_{bb}''+\widetilde{H}_{\beta c}Q_{cb}'')\\
   Q_{\gamma_1}^{\mathbb{T}}\widetilde{H}_{\gamma_1\alpha} Q_{\alpha}
   & Q_{\gamma_1}^{\mathbb{T}}\widetilde{H}_{\gamma_1\beta} Q_{l}
   & Q_{\gamma_1}^{\mathbb{T}}\widetilde{H}_{\gamma_1\gamma_1} Q_{\gamma_1}
   & Q_{\gamma_1}^{\mathbb{T}}(\widetilde{H}_{\gamma_1b}Q_{bb}''+\widetilde{H}_{\gamma_1c}Q_{cb}'')\\
   (Q')^{\mathbb{T}}\widetilde{H}_{b\alpha} Q_{\alpha}
   & (Q')^{\mathbb{T}}\widetilde{H}_{b\beta} Q_{l}
   & (Q')^{\mathbb{T}}\widetilde{H}_{b\gamma_1} Q_{\gamma_1}
   & (Q')^{\mathbb{T}}(\widetilde{H}_{bb} Q_{bb}''+\widetilde{H}_{bc} Q_{cb}'')\\
   \end{matrix}\right]
   \end{align*}
  where $Q_{\alpha}={\rm Diag}(Q_1,\ldots,Q_{l-1})$ and $Q_{\gamma_1}={\rm Diag}(Q_{l+1},\ldots,Q_r)$
  are the block diagonal orthogonal matrices. By the definitions of
  $\widehat{\Theta}_1,\widehat{\Theta}_2,\widehat{\Sigma}_1,\Sigma_2$,
  we write \eqref{temp-equa1} equivalently as
  \begin{align}\label{first-group}
    \widetilde{G}_{\alpha\alpha}+(\Omega_2)_{\alpha\alpha}\circ (\mathcal{X}(\widetilde{H}_1-\widetilde{G}_1))_{\alpha\alpha}=0,
    \qquad\qquad\qquad\qquad\nonumber\\
    \widetilde{G}_{\alpha\beta}+(\Omega_2)_{\alpha\beta}\circ \big(\mathcal{X}(\widetilde{H}_1\!-\!\widetilde{G}_1)\big)_{\alpha\beta}=0,\,
    \widetilde{G}_{\beta\alpha}+(\Omega_2)_{\beta\alpha}\circ \big(\mathcal{X}(\widetilde{H}_1\!-\!\widetilde{G}_1)\big)_{\beta\alpha}=0,\\
   \widetilde{G}_{\alpha\gamma_1}+(\Omega_1)_{\alpha\gamma_1}\circ\big(\mathcal{S}(\widetilde{H}_1\!-\!\widetilde{G}_1)\big)_{\alpha\gamma_1}
   +(\Omega_2)_{\alpha\gamma_1}\circ\big(\mathcal{X}(\widetilde{H}_1\!-\!\widetilde{G}_1)\big)_{\alpha\gamma_1}=0,\qquad\nonumber\\
   \widetilde{G}_{\gamma_1\alpha}+(\Omega_1)_{\gamma_1\alpha}\circ\big(\mathcal{S}(\widetilde{H}_1\!-\!\widetilde{G}_1)\big)_{\gamma_1\alpha}
   +(\Omega_2)_{\gamma_1\alpha}\circ\big(\mathcal{X}(\widetilde{H}_1\!-\!\widetilde{G}_1)\big)_{\gamma_1\alpha}=0\qquad\nonumber
   \end{align}
   and
  \begin{subequations}
  \begin{align}\label{equa1-WGH}
   \widetilde{G}_{b\alpha}+(\Omega_1)_{b\alpha}\circ (\widetilde{H}_{b\alpha}\!-\!\widetilde{G}_{b\alpha})=0,
   \qquad\qquad\qquad\qquad\\
    \label{equa2-WGH}
   \widetilde{G}_{\alpha b}Q_{bb}''+ \widetilde{G}_{\alpha c}Q_{cb}''
   +(\Omega_1)_{\alpha b}\circ\big[(\widetilde{H}_{\alpha b}-\widetilde{G}_{\alpha b})Q_{bb}''
   +(\widetilde{H}_{\alpha c}-\widetilde{G}_{\alpha c})Q_{cb}''\big]=0,\\
    \label{equa3-WGH}
       \Xi_1\circ \big(\mathcal{S}(Q_{l}^{\mathbb{T}}\widetilde{H}_{\beta\beta}Q_{l})\big)
     +(\Xi_1+\Xi_2)\circ \big(\mathcal{X}(Q_{l}^{\mathbb{T}}\widetilde{H}_{\beta\beta}Q_{l})\big)
    +\Xi_2\circ \big(\mathcal{S}(Q_{l}^{\mathbb{T}}\widetilde{G}_{\beta\beta}Q_{l})\big)=0,\\
    \label{equa4-WGH}
   \widetilde{H}_{\beta\gamma_1}=0,\,\widetilde{H}_{\gamma_1\beta}=0,\,\widetilde{H}_{\gamma_1\gamma_1}=0,\,
   \widetilde{H}_{b\beta}=0,\,\widetilde{H}_{b\gamma_1}=0,\qquad\qquad\\
    \label{equa5-WGH}
   \widetilde{H}_{\beta b}Q_{bb}''+\widetilde{H}_{\beta c}Q_{cb}''=0,\,
   \widetilde{H}_{\gamma_1 b}Q_{bb}''+\widetilde{H}_{\gamma_1 c}Q_{cb}''=0,\,
    \widetilde{H}_{b b}Q_{bb}''+\widetilde{H}_{b c}Q_{cb}''=0\quad
   \end{align}
   \end{subequations}
  where equalities \eqref{equa1-WGH} and \eqref{equa2-WGH} are using
  $(\Omega_1)_{b\alpha}=(\Omega_2)_{b\alpha}$ and the fact that the entries
  in each column of $(\Omega_1)_{b\alpha}$ are the same. Taking the limit $k\to\infty$
  to \eqref{WGHk-equa2}-\eqref{WGHk-equa3}, we get
  \begin{subequations}
  \begin{align}  \label{equa-WGHac}
   \widetilde{G}_{\alpha b}Q_{bc}''+ \widetilde{G}_{\alpha c}Q_{cc}''
   +(\Omega_3)_{\alpha c}\circ\big[(\widetilde{H}_{\alpha b}-\widetilde{G}_{\alpha b})Q_{bc}''
   +(\widetilde{H}_{\alpha c}-\widetilde{G}_{\alpha c})Q_{cc}''\big]=0,\\
  \label{equa-WGHbp}
  \widetilde{H}_{\beta_{+}b}Q_{bc}''+\widetilde{H}_{\beta_{+} c}Q_{cc}''=0,\qquad\qquad\qquad\qquad\\
  \label{equa-WGHb0m}
  \widetilde{H}_{\beta_0b}Q_{bc}''+\widetilde{H}_{\beta_0 c}Q_{cc}''=0,\,
   \widetilde{H}_{\beta_{-}b}Q_{bc}''+\widetilde{H}_{\beta_{-}c}Q_{cc}''=0,\,
   \widetilde{H}_{\gamma b}Q_{bc}''+\widetilde{H}_{\gamma c}Q_{cc}''=0.
  \end{align}
 \end{subequations}
 Notice that $(\Omega_1)_{\alpha b}=(\Omega_3)_{\alpha b}$ and they have the same entries in each row.
 Hence, equations \eqref{equa2-WGH} and \eqref{equa-WGHac} are equivalent to saying that
 \(
   \widetilde{G}_{\alpha, b\cup c}+(\Omega_1)_{\alpha,b\cup c}\circ
   (\widetilde{H}_{\alpha,b\cup c}-\widetilde{G}_{\alpha,b\cup c})=0.
 \)
 The equalities in \eqref{equa5-WGH} and \eqref{equa-WGHbp}-\eqref{equa-WGHb0m} are equivalent to
 saying that $\widetilde{H}_{\beta,b\cup c}=0$ and $\widetilde{H}_{\gamma,b\cup c}=0$.
 The three equalities, along with the equalities in \eqref{first-group}, \eqref{equa1-WGH}
 and \eqref{equa4-WGH}, can be compactly written as \eqref{equa1-normal} and \eqref{equa2-normal}.
 Finally, by the partition $(\beta_{+},\beta_{0},\beta_{-})$ of $\beta$, we may write
 $Q_l=[Q_{\beta_{+}}\ \ Q_{\beta_0}\ \ Q_{\beta_{-}}]\in\mathbb{O}^{|\beta|}$ with
 $Q_{\beta_0}\in\mathbb{O}^{|\beta|\times|\beta_0|}$. Then, taking the limit to \eqref{WGHk-equa4}
 yields that $Q_{\beta_0}^{\mathbb{T}}\widetilde{G}_{\beta\beta}Q_{\beta_0}\preceq 0$
 and $Q_{\beta_0}^{\mathbb{T}}\widetilde{H}_{\beta\beta}Q_{\beta_0}\succeq 0$.
 Together with \eqref{equa3-WGH}, it follows that $\widetilde{G}_{\beta\beta}$ and
 $\widetilde{H}_{\beta\beta}$ satisfy \eqref{equa3-normal}.
 To sum up, we achieve the desired result in \eqref{equa1-normal}-\eqref{equa3-normal}.

 \medskip
 \noindent
 ``$\Longleftarrow$.''  Write $\overline{\sigma}:=\sigma(\overline{Z})$.
 Notice that $\Pi_{\mathbb{B}}(X+Y)=Y$. From the SVD of $\overline{Z}=X+Y$,
 \begin{equation}\label{XSVD}
  X=\overline{U}_{\!\alpha}{\rm Diag}(\overline{\sigma}_{\alpha}\!-\!e_{\alpha})\overline{V}_{\!\alpha}^{\mathbb{T}}
  \ \ {\rm and}\
  Y=\overline{U}\!\left[\begin{matrix}
                   {\rm Diag}(e_{\alpha})& 0_{\alpha\beta}&0_{\alpha\gamma}&0_{\alpha c}\\
                   0_{\beta\alpha} &  {\rm Diag}(e_{\beta}) &  0_{\beta\gamma}&0_{\beta c}\\
                   0_{\gamma\alpha}&0_{\gamma\beta}  & {\rm Diag}(\overline{\sigma}_{\!\gamma})&0_{\gamma c}
                   \end{matrix}\right]\!\overline{V}^{\mathbb{T}}.
 \end{equation}
 Let $(G,H)$ satisfy \eqref{equa1-normal}-\eqref{equa3-normal}.
 Then there exist $Q\in\mathbb{O}^{|\beta|}$, $\Xi_1\in\mathcal{U}_{|\beta|}$
 and a partition $(\beta_{+},\beta_{0},\beta_{-})\in\mathscr{P}(\beta)$ such that
 $\Xi_1$ and the associated $\Xi_2$ take the form \eqref{Xi1-matrix}-\eqref{Xi2-matrix},
 and
 \begin{subequations}
 \begin{align}\label{equa1-bbeta}
  \Xi_1\circ (Q^{\mathbb{T}}\widetilde{H}_{\beta\beta}Q)
   +\Xi_2\circ \mathcal{S}(Q^{\mathbb{T}}\widetilde{G}_{\beta\beta}Q)+\Xi_2\circ \mathcal{X}(Q^{\mathbb{T}}\widetilde{H}_{\beta\beta}Q)=0,\\
    Q_{\beta_0}^{\mathbb{T}}\widetilde{G}_{\beta\beta}Q_{\beta_0}\preceq 0\ \ {\rm and}\ \
   Q_{\beta_0}^{\mathbb{T}}\widetilde{H}_{\beta\beta}Q_{\beta_0}\succeq 0.\qquad\quad
   \label{equa2-bbeta}
 \end{align}
 \end{subequations}
 Since $\Xi_1\in\mathcal{U}_{|\beta|}$, there exists a sequence $\{z^k\}\subseteq\mathbb{R}_{>}^{|\beta|}$
 converging to $e_{\beta}$ such that $\Xi_1=\lim_{k\to\infty}D(z^k)$. Without loss of generality,
 we may assume that for all $k$,
 \[
   z_i^k>1\quad \forall i\in\beta_{+},\quad z_i^k=1\quad \forall i\in\beta_0\ \ {\rm and}\ \
   0<z_i^k<1\quad \forall i\in\beta_{-}.
 \]
 For each $k$, we construct the matrices $X^k$ and $Y^k$ as follows:
 \begin{subequations}
 \begin{align*}
   X^k&=\widehat{U}\left[\begin{matrix}
       {\rm Diag}(\overline{\sigma}_{\alpha}\!-\!e_{\alpha})&0_{\alpha\beta_{+}}&0_{\alpha\beta_{0}}&0_{\alpha\beta_{-}}
       & 0_{\alpha\gamma}&0_{\alpha c}\\
       0_{\beta_{+}\alpha}&{\rm Diag}(z_{\beta_{+}}^k\!-\!e_{\beta_{+}})&0_{\beta_{+}\beta_{0}}&0_{\beta_{+}\beta_{-}}& 0_{\beta_{+}\gamma}&0_{\beta_{+}c}\\
       0_{\beta_{0}\alpha}&0_{\beta_{0}\beta_{+}}& 0_{\beta_{0}\beta_{0}}
       &0_{\beta_{0}\beta_{-}}& 0_{\beta_{0}\gamma}& 0_{\beta_{0}c}\\
       0_{\beta_{-}\alpha}&0_{\beta_{-}\beta_{+}}&0_{\beta_{-}\beta_{0}}
       & 0_{\beta_{-}\beta_{-}}&0_{\beta_{-}\gamma}&0_{\beta_{-}c}\\
       0_{\gamma\alpha}&0_{\gamma\beta_{+}}&0_{\gamma\beta_{0}}&0_{\gamma\beta_{-}}& 0_{\gamma\gamma}& 0_{\gamma c}\\
       \end{matrix}\right]\widehat{V}^{\mathbb{T}},\\
  Y^k&=\widehat{U}\left[\begin{matrix}
       {\rm Diag}(e_{\alpha})&0_{\alpha\beta_{+}}&0_{\alpha\beta_{0}}&0_{\alpha\beta_{-}}& 0_{\alpha\gamma}& 0_{\alpha c}\\
       0_{\beta_{+}\alpha}&{\rm Diag}(e_{\beta_{+}})&0_{\beta_{+}\beta_{0}}&0_{\beta_{+}\beta_{-}}& 0_{\beta_{+}\gamma}& 0_{\beta_{+}c}\\
       0_{\beta_{0}\alpha}&0_{\beta_{0}\beta_{+}}& {\rm Diag}(e_{\beta_{0}})&0_{\beta_{0}\beta_{-}}& 0_{\beta_{0}\gamma}& 0_{\beta_{0}c}\\
       0_{\beta_{-}\alpha}&0_{\beta_{-}\beta_{+}}&0_{\beta_{-}\beta_{0}}& {\rm Diag}(z_{\beta_{-}}^k)&0_{\beta_{-}\gamma}&0_{\beta_{-}c}\\
       0_{\gamma\alpha}&0_{\gamma\beta_{+}}&0_{\gamma\beta_{0}}&0_{\gamma\beta_{-}}& {\rm Diag}(\overline{\sigma}_{\gamma})&0_{\gamma c}\\
       \end{matrix}\right]\widehat{V}^{\mathbb{T}}
 \end{align*}
 \end{subequations}
 where $\widehat{U}=[\overline{U}_{\!\alpha}\ \ \overline{U}_{\!\beta}Q\ \ \overline{U}_{\!\gamma}]$
 and $\widehat{V}=[\overline{V}_{\!\alpha}\ \ \overline{V}_{\!\beta}Q\ \ \overline{V}_{\!\gamma}\ \ \overline{V}_{\!c}]$.
 It is immediate to see that $\Pi_{\mathbb{B}}(X^k+Y^k)=Y^k$ for each $k$,
 which by Lemma \ref{graph-subdiff} shows that $(X^k,Y^k)\in{\rm gph}\,\partial\|\cdot\|_*$.
 Also, comparing with \eqref{XSVD}, we have that $(X^k,Y^k)$ converges to $(X,Y)$.
 For each $k$, we write $Z^k=X^k+Y^k$ and define the matrices
 $\Theta_1^k,\Theta_2^k,\Sigma_1^k,\Sigma_2^k\in\mathbb{S}^m$ as in \eqref{Theta1k}-\eqref{Sigma2k}
 with $\sigma(Z^k)$. Observe that $\sigma(Z^k)=(\overline{\sigma}_{\!\alpha};z_{\beta_{+}}^k;\overline{\sigma}_{\beta_0};
 z_{\beta_{-}}^k;\overline{\sigma}_{\!\gamma})$. Then, we have that
 \begin{align}\label{limit1-Omega12}
    (\Omega_1^k)_{ij}=(\Omega_1)_{ij}\ \ {\rm and}\ \ ({\Omega}_2^k)_{ij}=(\Omega_2)_{ij}
    \ \ {\rm for}\ (i,j)\in(\alpha\cup\beta_{0}\cup\gamma)\times(\alpha\cup\beta_{0}\cup\gamma),\\
    \lim_{k\to\infty}({\Omega}_1^k)_{ij}=(\Omega_1)_{ij},\,
    \lim_{k\to\infty}({\Omega}_2^k)_{ij}=(\Omega_2)_{ij}
    \ \ {\rm for}\ (i,j)\in\{1,2,\ldots,m\}\times(\beta_{+}\cup\beta_{-}).
    \label{limit2-Omega12}
  \end{align}
  Let $\Omega_3^k$ be defined by \eqref{Omega3} with $\sigma(Z^k)$.
  Clearly, $(\Omega_3^k)_{ij}=(\Omega_3)_{ij}$ for $(i,j)\in\alpha\times c$.
  The rest is to construct a sequence $\{(G^k,H^k)\}$ converging to $(G,H)$
  such that $(G^k,H^k)\in\mathcal{N}_{{\rm gph}\,\partial\|\cdot\|_*}^{\pi}(X^k,Y^k)$ for each $k$.
  For this purpose, we shall define $\widehat{G}^k,\widehat{H}^k\in\mathbb{R}^{m\times n}$. Let
  \begin{equation}\label{equa0-whGk}
   (\widehat{G}^k)_{ij}\!:=\widetilde{G}_{ij}\ {\rm and}\
   (\widehat{H}^k)_{ij}\!:=\widetilde{H}_{ij}\ {\rm for}\
   (i,j)\in(\alpha\cup\beta_{0}\cup\gamma)\!\times\!(\alpha\cup\beta_{0}\cup\gamma)\ {\rm or}\
   \alpha\times c.
  \end{equation}
  For $(i,j)\notin(\alpha\cup\beta_{0}\cup\gamma)\times(\alpha\cup\beta_{0}\cup\gamma)$
  or $\alpha\times c$, we define $(\widehat{G}^k)_{ij}$ and $(\widehat{H}^k)_{ij}$
  as below, where $\widehat{G}_1^k$ and $\widehat{H}_1^k$ are
  the matrices consisting of the first $m$ columns of $\widehat{G}^k$ and $\widehat{H}^k$.

  \medskip
  \noindent
  {\bf Case 1:} $(i,j)$ or $(j,i)\in\alpha\times\beta_{+}$. In this case,
  we let $\widehat{H}^k_{ij}:=\widetilde{H}_{ij}$ for each $k$ and define
  \[
    \widehat{G}^k_{ij}:=\frac{(\Omega_2^k)_{ij}}{(\Omega_2^k)_{ij}-1} \big[\mathcal{X}(\widetilde{H}_1)\big]_{ij}.
  \]
  Notice that $(\Omega_2^k)_{ij}=(\Omega_2^k)_{ji}$. Then we have
  $\widehat{G}^k_{ij}=-\widehat{G}^k_{ji}$ for each $k$, which implies that
  \begin{equation}\label{equa1-whGk}
  \widehat{G}_{ij}^k\!+\!(\Omega_2^k)_{ij}[\mathcal{X}(\widehat{H}_1^k-\widehat{G}_1^k)]_{ij}=0.
  \end{equation}
%-------------------------------------------------------------------------------------------Case2
  {\bf Case 2:} $(i,j)$ or $(j,i)\in\alpha\times\beta_{-}$.
  In this case, we let $\widehat{H}^k_{ij}:=\widetilde{H}_{ij}$ for each $k$ and define
  \[
    \widehat{G}^k_{ij}:=\frac{(\Omega_2^k)_{ij}}{(\Omega_2^k)_{ij}-1} \big[\mathcal{X}(\widetilde{H}_1)\big]_{ij}
    +\frac{(\Omega_1^k)_{ij}}{(\Omega_1^k)_{ij}-1} \big[\mathcal{S}(\widetilde{H}_1)\big]_{ij}.
  \]
  Then  $\big[\mathcal{S}(\widehat{G}_1^k)\big]_{ij}=-\frac{(\Omega_1^k)_{ij}}{(\widetilde{\Omega}_1^k)_{ij}}
  \big[\mathcal{S}(\widehat{H}_1^k)\big]_{ij}$ and $\big[\mathcal{X}(\widehat{G}_1^k)\big]_{ij}=-\frac{(\Omega_2^k)_{ij}}{(\widetilde{\Omega}_2^k)_{ij}}
  \big[\mathcal{X}(\widehat{H}_1^k)\big]_{ij}$ by using the symmetry of $\Omega_2^k$ and $\Omega_1^k$.
  Consequently, for each $k$, it holds that
  \begin{equation}\label{equa2-whGk}
  (\Omega_1^k)_{ij}[\mathcal{S}(\widehat{H}^k)]_{ij}+
  (\widetilde{\Omega}_1^k)_{ij}[\mathcal{S}(\widehat{G}^k)]_{ij}=0,\,
  (\Omega_2^k)_{ij}[\mathcal{X}(\widehat{H}^k)]_{ij}
  + (\widetilde{\Omega}_2^k)_{ij}[\mathcal{X}(\widehat{G}_1^k)]_{ij}=0.
  \end{equation}
%----------------------------------------------------------------------------------------------Case3
  {\bf Case 3:} $(i,j)$ or $(j,i)\in(\beta_{+}\cup\beta_{0})\times\beta_{+}$. For each $k$,
  we let $\widehat{G}^k_{ij}:=Q_i^{\mathbb{T}}\widetilde{G}_{\beta\beta}Q_j$ and
  \[
    \widehat{H}^k_{ij}:=\frac{(\Omega_2^k)_{ij}-1}{(\Omega_2^k)_{ij}} \big[\mathcal{X}(\widehat{G}_1^k)\big]_{ij}.
  \]
  Notice that $\widehat{G}_{ij}^k=-\widehat{G}_{ji}^k$ implied by equation \eqref{equa1-bbeta}. Then, we immediately have that
  \[
    (\widehat{G}^k)_{ij}-(\Omega_2^k)_{ij}\big[\mathcal{X}(\widehat{G}_1^k)\big]_{ij}
    +(\Omega_2^k)_{ij}\big[\mathcal{X}(\widehat{H}_1^k)\big]_{ij}=0.
  \]
%----------------------------------------------------------------------------------------------Case 4
  {\bf Case 4:} $(i,j)$ or $(j,i)\in\beta_{+}\times\beta_{-}$. For each $k$,
  we let $\widehat{G}^k_{ij}:=Q_i^{\mathbb{T}}\widetilde{G}_{\beta\beta}Q_j$ and define
  \[
    \widehat{H}^k_{ij}:=\frac{(\Omega_2^k)_{ij}-1}{(\Omega_2^k)_{ij}} \big[\mathcal{X}(\widehat{G}_1^k)\big]_{ij}
    +\frac{(\Omega_1^k)_{ij}-1}{(\Omega_1^k)_{ij}}\big[\mathcal{S}(\widehat{G}_1^k)\big]_{ij}.
  \]
  Then $\big[\mathcal{S}(\widehat{H}_1^k)\big]_{ij}=-\frac{(\widetilde{\Omega}_1^k)_{ij}}{(\Omega_1^k)_{ij}}
  \big[\mathcal{S}(\widehat{G}_1^k)\big]_{ij}$ and $\big[\mathcal{X}(\widehat{H}_1^k)\big]_{ij}=-\frac{(\widetilde{\Omega}_2^k)_{ij}}{(\Omega_2^k)_{ij}}
  \big[\mathcal{X}(\widehat{G}_1^k)\big]_{ij}$, and hence
  \begin{equation}\label{equa4-whGk}
  (\Omega_1^k)_{ij}[\mathcal{S}(\widehat{H}^k)]_{ij}+
  (\widetilde{\Omega}_1^k)_{ij}[\mathcal{S}(\widehat{G}^k)]_{ij}=0,\,
  (\Omega_2^k)_{ij}[\mathcal{X}(\widehat{H}^k)]_{ij}
  + (\widetilde{\Omega}_2^k)_{ij}[\mathcal{X}(\widehat{G}_1^k)]_{ij}=0.
 \end{equation}
%--------------------------------------------------------------------------------------------Case 5
  {\bf Case 5:} $(i,j)$ or $(j,i)\in\beta_{+}\times\gamma$. In this case,
   we let $\widehat{G}^k_{ij}:=\widetilde{G}_{ij}$ for each $k$ and define
  \[
    \widehat{H}^k_{ij}:=\frac{(\Omega_2^k)_{ij}-1}{(\Omega_2^k)_{ij}} \big[\mathcal{X}(\widehat{G}_1^k)\big]_{ij}
    +\frac{(\Omega_1^k)_{ij}-1}{(\Omega_1^k)_{ij}}\big[\mathcal{S}(\widehat{G}_1^k)\big]_{ij}.
  \]
  Then, using the same arguments as those for Case 4, we obtain that
  \begin{equation}\label{equa5-whGk}
  (\Omega_1^k)_{ij}[\mathcal{S}(\widehat{H}^k)]_{ij}+(\widetilde{\Omega}_1^k)_{ij}[\mathcal{S}(\widehat{G}^k)]_{ij}=0,\,
  (\Omega_2^k)_{ij}[\mathcal{X}(\widehat{H}^k)]_{ij}
  + (\widetilde{\Omega}_2^k)_{ij}[\mathcal{X}(\widehat{G}_1^k)]_{ij}=0.
  \end{equation}
%----------------------------------------------------------------------------------------------Case 6
  {\bf Case 6:} $(i,j)\in\beta_{+}\times c$. For each $k$, let $\widehat{G}^k_{ij}:=\widetilde{G}_{ij}$
  and
  \(
    \widehat{H}^k_{ij}:=\frac{(\Omega_3^k)_{ij}-1}{(\Omega_3^k)_{ij}}\widetilde{G}_{ij}.
  \)
  Then,
  \begin{equation}\label{equa6-whGk}
   (E_{\beta_{+} c}-(\Omega_3^k)_{\beta_{+} c})\circ(\widehat{G}^k)_{\beta_{+} c}
   +(\Omega_3^k)_{\beta_{+} c}\circ (\widehat{H}^k)_{\beta_{+} c}=0.
  \end{equation}
%-----------------------------------------------------------------------------------------
  {\bf Case 7:} $(i,j)\in(\beta_{0}\cup\beta_{-}\cup\gamma)\times (\beta_{-}\cup c)$
  or $(i,j)\in\beta_{-}\times(\gamma\cup\beta_{0})$. Now for each $k$,
  let $\widehat{G}^k_{ij}:=\widetilde{G}_{ij}$ and $\widehat{H}^k_{ij}:=0$.

  \medskip

  For each $k$, let $G^k\!:=\!\widehat{U}\widehat{G}^k\widehat{V}^{\mathbb{T}}$
  and $H^k\!:=\widehat{U}\widehat{H}^k\widehat{V}^{\mathbb{T}}$.
  By the construction of $\widehat{G}^k$ and $\widehat{H}^k$,
  \begin{align}\label{equa1-final}
    \Theta_1^k\circ \mathcal{S}(\widehat{U}^{\mathbb{T}}H^k\widehat{V}_1)+\Theta_2^k\circ \mathcal{S}(\widehat{U}^{\mathbb{T}}G^k\widehat{V}_1)
    +\Sigma_1^k\circ \mathcal{X}(\widehat{U}^{\mathbb{T}}H^k\widehat{V}_1)+ \Sigma_2^k\circ \mathcal{X}(\widehat{U}^{\mathbb{T}}G^k\widehat{V}_1)=0,
    \nonumber\\
     (\widehat{U}^{\mathbb{T}}G^k\widehat{V})_{\alpha c}\circ (E_{\alpha c}-(\Omega_3^k)_{\alpha c})
     +(\widehat{U}^{\mathbb{T}}H^k\widehat{V})_{\alpha c}\circ (\Omega_3^k)_{\alpha c}=0,\qquad\qquad\\
    % \label{equa2-final}
    (\widehat{U}^{\mathbb{T}}G^k\widehat{V})_{\beta_{+} c}\circ (E_{\beta_{+} c}-(\Omega_3^k)_{\beta_{+} c})
     +(\widehat{U}^{\mathbb{T}}H^k\widehat{V})_{\beta_{+}c}\circ (\Omega_3^k)_{\beta_{+} c}=0,\qquad\qquad\nonumber\\
    % \label{equa4-final}
     (\widehat{U}^{\mathbb{T}}H^k\widehat{V})_{\beta_{0}c}=0,\ \ (\widehat{U}^{\mathbb{T}}H^k\widehat{V})_{\beta_{-}\cup\gamma,c}=0.\qquad\qquad\qquad\qquad\nonumber
  \end{align}
  In addition, from equations \eqref{equa0-whGk} and \eqref{equa2-bbeta}, for each $k$ we have that
  \begin{equation}\label{equa2-final}
   (\widehat{U}^{\mathbb{T}}G^k\widehat{V})_{\beta\beta}
    = Q_{\beta_0}^{\mathbb{T}}\widetilde{G}_{\beta\beta}Q_{\beta_0}\preceq 0\ \ {\rm and}\ \
   (\widehat{U}^{\mathbb{T}}H^k\widehat{V})_{\beta\beta}
   = Q_{\beta_0}^{\mathbb{T}}\widetilde{H}_{\beta\beta}Q_{\beta_0}\succeq 0.
  \end{equation}
  From \eqref{equa1-final}-\eqref{equa2-final} and Theorem \ref{proxnormal-theorem},
  $(G^k,H^k)\in\mathcal{N}_{{\rm gph}\,\partial\|\cdot\|_*}^{\pi}(X^k,Y^k)$ for each $k$.
  Since $(G,H)$ satisfies \eqref{equa1-normal}-\eqref{equa2-normal},
  from the construction of $(\widehat{G}^k,\widehat{H}^k)$ and
  equation \eqref{limit1-Omega12}-\eqref{limit2-Omega12}, we have
  \[
    \lim_{k\to\infty}\widehat{U}^{\mathbb{T}}G^k\widehat{V}=\lim_{k\to\infty}\widehat{G}^k=\widehat{U}^{\mathbb{T}}G\widehat{V}
    \ \ {\rm and}\ \
   \lim_{k\to\infty}\widehat{U}^{\mathbb{T}}H^k\widehat{V}=\lim_{k\to\infty}\widehat{H}^k=\widehat{U}^{\mathbb{T}}H\widehat{V}.
  \]
  This implies that $\lim_{k\to\infty} (G^k,H^k)=(G,H)$.
  Together with $\lim_{k\to\infty}(X^k,Y^k)=(X,Y)$
  and $(G^k,H^k)\in\mathcal{N}_{{\rm gph}\,\partial\|\cdot\|_*}^{\pi}(X^k,Y^k)$,
  the converse conclusion follows. \hfill$\Box$\medskip
 \end{document}